\documentclass[review,3p,square]{elsarticle}
\usepackage{amssymb}
\usepackage{amsmath}
\usepackage{amsthm}
\usepackage{stmaryrd}
\usepackage{bbm}
\usepackage{graphicx}
\usepackage{natbib}
\usepackage{pifont}
\usepackage{float}
\usepackage{CJKutf8}

\biboptions{sort&compress}

\usepackage{enumitem}
\journal{J. Theoret. Probab.}
\usepackage{fancyhdr}
\usepackage{hyperref}
\usepackage{bookmark}

\hypersetup{%
	pdftitle={BSDEs},%
	pdfsubject={BSDEs},%
	pdfauthor={Chuang Gu, Yan Wang, ShengJun FAN},%
	pdfkeywords={BSDEs},%
	pdfstartview=FitH,%
	CJKbookmarks=true,%
	bookmarksnumbered=true,%
	bookmarksopen=true,%
	colorlinks=true, linkcolor=blue, urlcolor=blue, citecolor=blue, %
}

\usepackage{cleveref}

\newtheorem{theorem}{Theorem}[section]

\newtheorem{proposition}[theorem]{Proposition}

\theoremstyle{definition}

\newtheorem{example}[theorem]{Example}
\newtheorem{remark}[theorem]{Remark}
\numberwithin{equation}{section}

\newtheorem{lemma}{Lemma}[section]

\journal{arXiv}
\begin{document}
\begin{CJK}{UTF8}{gbsn}
	
	\begin{frontmatter}
		
		\title{{Existence, uniqueness and comparison theorem on  unbounded solutions of general time interval BSDEs with sub-quadratic generators}\tnoteref{found}}
		\tnotetext[found]{Partially supported by National Natural Science Foundation of China (No. 12171471),  the Natural Science Foundation of Jiangsu Province (No. BK20231057), and  the Graduate Innovation Program of China University of Mining and Technology (No. 2023WLJCRCZL143).
			\vspace{0.2cm}}

		\author{Chuang Gu}
		\ead{guchuang1026@163.com}
		
		\author{Yan Wang}
		\ead{wangyan\_shuxue@163.com}
		
		\author{Shengjun Fan\corref{cor}}
		\ead{shengjunfan@cumt.edu.cn}
		
		\cortext[cor]{Corresponding author}
		
		\address {School of Mathematics,
			China University of Mining and Technology,
			Jiangsu 221116, P.R. China}
		 \vspace{-0.5cm}
		
		\vspace{0.2cm}
		\begin{abstract}
			This paper is devoted to the existence, uniqueness and comparison theorem on unbounded solutions of one-dimensional backward stochastic differential equations (BSDEs) with sub-quadratic generators, where the  terminal time is allowed to be finite or infinite. We first establish  existence of the unbounded solutions for this kind of BSDEs with generator $g$  satisfying a time-varying one-sided  linear growth  in the first unknown variable  $y$ and a time-varying sub-quadratic growth  in the second unknown variable $z$. Then, the  uniqueness and comparison theorem of the  unbounded solutions for this kind of BSDEs are proved under a time-varying extended convexity assumption. These results generalized those obtained in \cite{12} to the general time interval BSDEs.  Finally, several  sufficient conditions ensuring that the  uniqueness holds are put forward and verified via some innovative ideas, which are explored at the first time even though for the case of finite time interval BSDEs.
			\vspace{0.2cm}
		\end{abstract}
		
		\begin{keyword}
			Existence and uniqueness; Unbounded solutions; Backward stochastic differential equation;    \\  \hspace*{1.75cm}  Sub-quadratic growth;  Comparison theorem; General time interval. \vspace{0.2cm}
			
			\MSC[2024] 60H10\vspace{0.2cm}
		\end{keyword}
		
	\end{frontmatter}
	\vspace{-0.4cm}

\section{Introduction } \label{Title:1}
	Throughout this paper,  let $d$ be a positive integer and $(B_{t})_{t\geq 0}$ a standard $d$-dimensional Brownian motion defined on a complete probability space $(\Omega,\mathcal{F},\mathbb{P})$,  $(\mathcal{F}_{t})_{t\geq 0}$  the completed natural $\sigma$-algebra generated by $(B_{t})_{t\geq 0}$ and $\mathcal{F}_{T}:=\mathcal{F}$, where the  terminal time  $T$ is a positive real or $T=+\infty$. Let $|x|$ represent the Euclidian norm for $x\in \mathbb{R}^{d}$ and $x\cdot y$ the usual scalar inner product for $x,y\in \mathbb{R}^{d}$.  We consider the following one-dimensional backward stochastic differential equation (BSDE for short):
	\begin{equation}\label{eq1.1}
		Y_{t}=\xi +\int_{t}^{T}g(s,Y_{s},Z_{s})ds-\int_{t}^{T}Z_{s}\cdot dB_{s},~~~~~t\in [0,T], \tag{1.1}
	\end{equation}
	where the terminal value $\xi$ is an $\mathcal{F}_{T}$-measurable random variable, and the generator  $g(\omega, t, y, z): \Omega\times [0,T]\times \mathbb{R} \times \mathbb{R}^{d} \mapsto \mathbb{R}$ is an $(\mathcal{F}_{t})$-progressively measurable stochastic process for each $(y,z)$ and almost everywhere continuous  in $(y,z)$. The triple $(\xi, T, g)$ is called the parameters of BSDE $(\ref{eq1.1})$ and  $(Y_{t},Z_{t})_{t\in[0,T]}$ is a solution of BSDE $(\ref{eq1.1})$, which is a pair of $(\mathcal{F}_{t})$-progressively measurable processes taking values in $\mathbb{R}\times \mathbb{R}^{d}$  and verifying $(\ref{eq1.1})$. BSDE (\ref{eq1.1}) is usually denoted by BSDE ($\xi,g$).	General nonlinear BSDEs were initially studied in 1990 by Pardoux and Peng \cite{20}, which  proved an existence and uniqueness result for $L^{2}$ solutions of multidimensional BSDEs with square-integrable parameters and uniformly Lipschitz continuous generators. Since then, BSDEs have  captured much attention from a plenty of scholars.  These equations have also gradually turned into important  tools in various fields such as partial differential equations, mathematical finance, stochastic control, nonlinear mathematical expectation and so on.  The interested readers can see \cite{5,8,9,12,16,17,18,21,22,23} for more details.

	In this paper, we are concerned with  BSDE (\ref{eq1.1}) with generator $g$ satisfying the following condition:
	\begin{equation}\label{eq1.2}
		\tag{1.2}
		{\rm{d}}\mathbb{P}\times {\rm{d}}t-a.e.,~~~|g(\omega,t,y,z)|\leq f_{t}(\omega)+\beta(t)|y|+\gamma(t)|z|^{\alpha},~~~\forall~(y,z)\in \mathbb{R} \times \mathbb{R}^{d},
	\end{equation}
where $(f_{t})_{t\in[0,T]}$ is a nonnegative $(\mathcal{F}_{t})$-progressively measurable process, $\beta(\cdot)$ and $\gamma(\cdot)$ are two nonnegative deterministic functions depending only on the time variable $t$, and $\alpha\in(1,2)$.  It is well known that  $L^{2}$ solutions,  $L^{p}~(p>1)$ solutions,  $L^{1}$ solutions and bounded solutions for general time interval BSDEs have been extensively studied. For example, under the condition that the  generator $g$ satisfies (\ref{eq1.2}) with $\alpha=1$, Chen and Wang \cite{4} first proposed and investigated  existence and uniqueness  for the $L^{2}$ solution of a general time interval BSDE and Fan \cite{10}, Fan and Jiang \cite{14}, Zong and Hu \cite{25,26} established  existence and uniqueness for the $L^{p}~(p>1)$ solution of a general time interval  BSDE. Furthermore,  Fan \cite{10,11} addressed existence and uniqueness for the $L^{1}$ solution of a general time interval BSDE when generator $g$ satisfies  (\ref{eq1.2}) with $\alpha\in(0,1)$.   Fan \cite{10}, Luo and Fan \cite{19}  studied  existence and uniqueness for the bounded solution of a general time interval BSDE with a so-called quadratic generator $g$ satisfying  (\ref{eq1.2}) with $\alpha=2$, see  Kobylanski \cite{18} for the case of finite time interval  BSDEs. In particular, we would like to mention that Bahlali et al \cite{1} proposed an existence and uniqueness result for the $L^{2}$ solution  of a finite time interval  quadratic BSDE with generator $g$ taking the form $f(y)|z|^{2}$, where $f$ is required to be globally integrable on $\mathbb{R}$ so that   $f$ being  a positive constant is not sufficient to ensure the existence of a solution. The existence and uniqueness  of the $L^{p}~(p>1)$ solutions for this kind of quadratic BSDEs was further investigated  in Yang \cite{24}, where $f$ is globally integrable on $\mathbb{R}$ and bounded on any compact subset of $\mathbb{R}$.

On the other hand, Briand and Hu \cite{2}  put forward  the localization procedure  to obtain  existence for the unbounded solution of a finite time interval  quadratic BSDE with  terminal conditions $(\xi,f_{\cdot})$ admitting only a certain exponential moment, which is the weakest integrability condition  to  guarantee  existence  of the  unbounded solution. The uniqueness of the unbounded solution was established in  Briand and Hu \cite{3} for a finite time interval quadratic BSDE with convex generators and unbounded terminal conditions $(\xi,f_{\cdot})$  admitting every exponential moments  by utilizing the $\theta$-difference technique. Some further investigations can be found in Delbaen et al \cite{6,7} and Fan et al \cite{13} for the unbounded solution of a finite time interval quadratic BSDE.
Recently, by virtue of the localization procedure, Fan and Hu \cite{12} explored  existence  for the unbounded solution of a finite time interval BSDE with  terminal conditions $(\xi,f_{\cdot})$ admitting every sub-exponential moments when the generator $g$ satisfies  (\ref{eq1.2}) with $\alpha\in(1,2)$ and both $\beta(\cdot)$ and $\gamma(\cdot)$ are nonnegative constants,  and used the $\theta$-difference technique to establish uniqueness for the unbounded solution under an extended convexity condition of the sub-quadratic generator $g$, which is typically satisfied by  a locally Lipschitz perturbation of a proper convex function.  It should be noted that all of the above studies mentioned in this paragraph  have been made for a finite time interval  BSDE. A question arises naturally: do these results hold for a general time interval BSDE? The present paper will focus on whether the results of \cite{12} can be extended to a general time interval BSDE.

More especially, this paper is devoted to  existence, uniqueness and comparison theorem  for the unbounded solution of a general time interval BSDE with  sub-quadratic generator $g$ satisfying (\ref{eq1.2}) with $\alpha\in(1,2)$ and  terminal conditions $(\xi,f_{\cdot})$ admitting every  sub-exponential moments. Our results generalize the corresponding ones reported in Fan and Hu \cite{12} to the infinite time interval case. Firstly, in order to obtain  existence of the  unbounded solution to a general time interval BSDE, two new time-varying assumptions (EX1) and (EX2) on the generator $g$ are introduced, see the beginning of Section \ref{Title:3} for more details. They extend assumption (H1$^{\prime\prime}$) of Fan and Hu \cite{12} to the infinite time interval case. We would emphasis that the two deterministic functions $\beta(\cdot)$ and $\gamma(\cdot)$ in (EX1) and (EX2) are required to satisfy $\beta(\cdot)\in L^{1}([0,T];\mathbb{R}_{+})$ and $\gamma(\cdot)\in L^{\frac{2}{2-\alpha}}([0,T];\mathbb{R}_{+})$, which may be the most appropriate integrability conditions for them in order to guarantee existence of the unbounded solution. When $\alpha=1$, $\frac{2}{2-\alpha}$ equals to 2, this is just the case in  Chen and Wang \cite{4}, Fan \cite{10} and Zong and Hu \cite{25,26}. Our whole idea of the proof is first to establish a uniform a priori estimate on the first process $Y_{\cdot}$ in  any bounded solution $(Y_{\cdot},Z_{\cdot})$ of BSDE ($\xi,g$) by constructing a proper test function $\psi(s,x)$ and applying It$\rm{\hat{o}}$-Tanaka$^{\prime}$s formula to $\psi(s,\hat{Y}_{s})$, where $\hat{Y}_{\cdot}$ is defined in (\ref{eq3.4}), and then to verify   existence for the unbounded solution of BSDE ($\xi,g$) by using the  localization procedure proposed in Briand and Hu \cite{2}, see the first two steps of the proof of Theorem \ref{The:3.2} in Section \ref{Title:3} for more details. Secondly, in order to study a comparison theorem for the unbounded solution of a general time interval BSDE, we put forward a new time-varying extended convexity assumption (UN), see the beginning of Section \ref{Title:4.1} for details. It should be mentioned that due to the presence of the term $\gamma(t)\left[{\rm{ln}}(e+|z_{2}|)\right]^{\frac{\alpha^{*}}{2}}$, the assumption (UN)  not only extends  assumption (H2$^{\prime}$) of Fan and Hu \cite{12} to the infinite time interval case, but  also is strictly weaker than it even though for the finite time interval case. Our main idea of the proof is to use the $\theta$-difference technique put forward in Briand and Hu \cite{3}. However, due to the weaker assumption (UN), a new difficulty arises naturally and is successfully overcome, see the first step of the proof of Theorem \ref{The:4.2} in Section \ref{Title:4} for more details. Finally, we put forward several weaker sufficient conditions ensuring that the time-varying extended convexity assumption (UN) holds, which may allow a more general perturbation of a proper convex function than that reported in Fan and Hu \cite{12}, see Propositions \ref{Pro:2} and \ref{Pro:3}  together with  Examples \ref{Exa:2} and \ref{Exa:4} for more details. Some of these conditions are explored at the first time even though for the case of finite time interval BSDEs.

The remaining of this paper is organized as follows. In Section $\ref{Title:2}$, we introduce some  notations and an important lemma to be used later. In Section $\ref{Title:3}$, we put forward and prove an existence result for the unbounded solutions of general time interval BSDEs  with sub-quadratic generators. The main result is stated in Theorem \ref{The:3.2},  which generalizes  Theorem 2.2 in Fan and Hu \cite{12} to the case of general time interval BSDEs. In Section $\ref{Title:4}$, we  establish a  comparison theorem  for the unbounded solutions of general time interval BSDEs with sub-quadratic generators. This result is stated in Theorem \ref{The:4.2}, which can be seen as an extension of Theorem 3.2 in Fan and Hu \cite{12}. An existence and uniqueness result  of the unbounded solutions for this kind of BSDEs was further given  in Theorem \ref{The:4.3}, which generalizes Theorem 3.9 of Fan and Hu \cite{12} to the infinite time interval case.  In Section $\ref{Title:5}$, we put forward and prove three sufficient conditions ensuring that the uniqueness holds (see Propositions \ref{Pro:1}, \ref{Pro:2} and \ref{Pro:3}). Some remarks and  examples are also provided in the last three sections to illustrate the preceding results, see Remark \ref{Re:3.2} in Section $\ref{Title:3}$	together with Remark \ref{Re:4.1} in Section $\ref{Title:4}$ as well as  Remarks \ref{Re:5.1}, \ref{Re:5.2} and \ref{Re:5.3}, Examples \ref{Exa:2} and \ref{Exa:4} in Section $\ref{Title:5}$.  Finally, three technical lemmas used to prove Propositions \ref{Pro:2} and \ref{Pro:3}  are provided in Appendix, which are interesting in their own rights.

\section{Preliminaries} \label{Title:2}
	In this section, we will introduce some  notations and an important lemma to be used later. Let $a\vee b$ denote the maximum of two reals $a$ and $b$, $a^{+}:=a\vee 0$ and $a^{-}:=(-a)^{+}$. For $\alpha \in(1,2)$, let $\alpha^{*}$ stand for the conjugate  of $\alpha$, that is, $1/\alpha +1/\alpha^{*}=1$ or
	$$
	\alpha^{*}:=\frac{\alpha}{\alpha-1}>2.
	$$
	Denote $\mathbb{R}_{+}:=\left[0, +\infty\right)$, $\mathbb{R}_{-}:=\left(-\infty, 0\right]$,  ${\bf 1}_{A}=1$ when $x\in A$ otherwise 0, and ${\rm{sgn}}(x):={\bf 1}_{x>0}-{\bf 1}_{x\leq 0}$. We recall that an $(\mathcal{F}_{t})$-progressively measure scalar process $(X_{t})_{t\in[0,T]}$ belongs to class (D) if the family of random variables $\left\{X_{\tau}\right\}$, $\tau$ running all $(\mathcal{F}_{t})$-stopping times valued in $[0,T]$, is uniformly integrable. For any real $p\geq 1$, we state the following spaces:\vspace{0.2cm}
	
	$\bullet$ $L^{\infty}(\Omega,\mathcal{F}_{T},\mathbb{P})$ represents the set of all $\mathcal{F}_{T}$-measurable bounded random variables.

    $\bullet$ $\mathcal{S}^{\infty}(0,T;\mathbb{R})$ represents the set of all $(\mathcal{F}_{t})$-progressively measure and continuous bounded processes.

	$\bullet$ $\mathcal{M}^{p}(0,T;\mathbb{R}^{d})$ represents the set of  all
    $(\mathcal{F}_{t})$-progressively measurable $\mathbb{R}^{d}$-valued processes $(Z_{t})_{t\in[0,T]}$ such that
	$$
	\parallel Z\parallel _{\mathcal{M}^{p}}:=\left({\mathbb{E}}\left[\left( \int_{0}^{T}|Z_{t}|^{2}dt\right)^{\frac{p}{2}} \right]\right)^{\frac{1}{p}}<+\infty.
	$$
	
	$\bullet$ $L^{p}([0,T];\mathbb{R}_{+})$ represents the set of all nonnegative deterministic  functions $u(\cdot)$ such that
	$$
	\int_{0}^{T}u^{p}(t)dt<+\infty.
	$$

Let  $\alpha \in (1,2)$,  $\beta(\cdot)\in L^{1}([0,T];\mathbb{R}_{+})$, $\gamma(\cdot)\in  L^{\frac{2}{2-\alpha}}([0,T];\mathbb{R}_{+})$, and $(f_{t})_{t\in[0,T]}$ be a nonnegative $(\mathcal{F}_{t})$-progressively measurable  process having sub-exponential moments of any order, i.e.,
	$$
	{\mathbb{E}}\left[{\rm exp}\left(p\left(\int_{0}^{T}f_{s}ds \right)^{\frac{2}{\alpha^{*}}} \right)\right]<+\infty,~~~\forall~ p>1,
	$$
	which will be used throughout the whole paper.

The following Lemma \ref{Le:2.1} can be regarded as a direct consequence of  Lemma 3.4, Theorems 3.2 and 3.3 in Fan \cite{10}, which will play an important role in the proof of our main results.
	
	\begin{lemma}\label{Le:2.1}
		Assume that $\xi\in L^{\infty}(\Omega,\mathcal{F}_{T},\mathbb{P})$, and the generator $g$  satisfies the following two assumptions:
		
		\begin{enumerate}
			\renewcommand{\theenumi}{(H\arabic{enumi})}
			\renewcommand{\labelenumi}{\theenumi}
			
			\item\label{H1}   {\rm{d}}$\mathbb{P}\times {\rm{d}}t-a.e.$, $g(\omega,t,\cdot,\cdot):\mathbb{R}\times\mathbb{R}^{d}\mapsto \mathbb{R}$ is continuous;
			
			\item\label{H2} There exists  a function $u(t)\in L^{1}([0,T];\mathbb{R}_{+})$  such that {\rm{d}}$\mathbb{P}\times {\rm{d}}t-a.e.$, for each $(y,z)\in\mathbb{R}\times \mathbb{R}^{d}$,
			$$
			|g(\omega,t,y,z)|\leq u(t),
			$$	
		\end{enumerate}	
		then BSDE $(\xi,g)$ has  a minimal bounded solution $(Y_{t},Z_{t})_{t\in[0,T]}$ in  $\mathcal{S}^{\infty}(0,T;\mathbb{R})\times \mathcal{M}^{2}(0,T;\mathbb{R}^{d})$.
		
		Furthermore, assume that $\xi, \xi^{\prime}\in L^{\infty}(\Omega,\mathcal{F}_{T},\mathbb{P})$, both generators $g$ and $g^{\prime}$ satisfy (H1) and (H2), $(Y_{t},Z_{t})_{t\in[0,T]}$ and $(Y_{t}^{\prime},Z_{t}^{\prime})_{t\in[0,T]}$ are, respectively, the minimal  bounded solutions of BSDE ($\xi,g$) and BSDE ($\xi^{\prime},g^{\prime}$) in the space of $\mathcal{S}^{\infty}(0,T;\mathbb{R})\times \mathcal{M}^{2}(0,T;\mathbb{R}^{d})$. If $ \mathbb{P}-a.s.$, $\xi\leq\xi^{\prime}$, and
		$$
		{\rm{d}}\mathbb{P}\times {\rm{d}}t-a.e.,~g(t,Y_{t}^{\prime},Z_{t}^{\prime})\leq g^{\prime}(t,Y_{t}^{\prime},Z_{t}^{\prime})~~({\rm{resp.}}~ g(t,Y_{t},Z_{t})\leq g^{\prime}(t,Y_{t},Z_{t})),
		$$
		then $\mathbb{P}-a.s.$, for each $t\in[0,T]$, $Y_{t}\leq Y_{t}^{\prime}$.
	\end{lemma}

	\section{Existence of the  unbounded solution}\label{Title:3}
Let us first introduce the following assumptions on the generator $g$, which can be compared with  assumption $(\rm{H1^{\prime\prime}})$ in Fan and Hu \cite{12}.
	
	\begin{enumerate}
		\renewcommand{\theenumi}{(EX\arabic{enumi})}
		\renewcommand{\labelenumi}{\theenumi}
		
		\item\label{EX1}  $g$ has a time-varying one-sided  linear growth in $y$ and a time-varying sub-quadratic growth in  $z$, i.e.,  ${\rm{d}}\mathbb{P}\times {\rm{d}}t-a.e.,$ for each $(y,z)\in \mathbb{R}\times \mathbb{R}^{d}$,
		$$
		{\rm{sgn}}(y)g(\omega,t,y,z)\leq f_{t}(\omega)+\beta(t)|y|+\gamma(t)|z|^{\alpha}.
		$$
		
		\item\label{EX2} {\rm{d}}$\mathbb{P}\times {\rm{d}}t-a.e.$, $g(\omega,t,\cdot,\cdot):\mathbb{R}\times\mathbb{R}^{d}\mapsto \mathbb{R}$ is continuous, and $g$ has a time-varying general growth in $y$ and a quadratic growth in $z$, i.e., there exists a real constant $c>0$ such that ${\rm{d}}\mathbb{P}\times {\rm{d}}t-a.e.,$ for each $(y,z)\in \mathbb{R}\times \mathbb{R}^{d}$,
		$$
		|g(\omega,t,y,z)|\leq f_{t}(\omega)+\beta(t)\psi (|y|)+c|z|^{2},
		$$
		where $\psi(\cdot) $ is a nonnegative, increasing, continous  function and satisfies $\psi(0)=0$.
	\end{enumerate}

	The following Theorem \ref{The:3.2} presents an existence theorem  for the unbounded solution of a general time interval BSDE  under assumptions  (EX1) and (EX2), which is the main result of this section.

	\begin{theorem}\label{The:3.2}
		Assume that  $\xi$ is a terminal value, and $g$ is a generator which  satisfies assumptions  (EX1) and (EX2). If
		$|\xi|+\int_{0}^{T}f_{t}dt$  has sub-exponential moments of any order, i.e.,
		\begin{equation}\label{eq3.1}
		{\mathbb{E}}\left[{\rm exp}\left(p\left(|\xi|+\int_{0}^{T}f_{s}ds \right)^{\frac{2}{\alpha^{*}}} \right)\right]<+\infty,~~~ \forall~ p>1,\tag{3.1}
		\end{equation}
		then BSDE $(\xi,g)$ admits a solution $(Y_{t},Z_{t})_{t\in[0,T]}$ such that $\mathbb{P}-a.s.$,
		\begin{equation}\label{eq3.2}
			\tag{3.2}
			\begin{aligned}
				&{\rm exp}\left(|Y_{t}|^{\frac{2}{\alpha^{*}}}\right)+{\mathbb{E}}_{t}\left[\int_{t}^{T}|Z_{s}|^{2}ds \right]\\
               &~\leq K {\mathbb{E}}_{t}\left[{\rm exp}\left(K\left(|\xi|+\int_{t}^{T}f_{s}ds \right)^{\frac{2}{\alpha^{*}}}\right) \right],~~~t\in[0,T],
			\end{aligned}
		\end{equation}
		where $K>0$ is a constant depending only on $(\alpha,T,\beta(\cdot),\gamma(\cdot))$.
		Furthermore, for each $p>1$, there exists a constant $K_{p}>0$ depending only on  $(\alpha,T,p,\beta(\cdot),\gamma(\cdot))$ such that $\mathbb{P}-a.s.$,
		\begin{equation}\label{eq3.3}
			\tag{3.3}
			\begin{aligned}	
				&{\mathbb{E}}\left[{\rm exp}\left(p\left(\sup\limits_{s\in[t,T]}|Y_{s}|\right) ^{\frac{2}{\alpha^{*}}} \right) \right] +{\mathbb{E}}\left[\left(\int_{t}^{T}|Z_{s}|^{2}ds \right)^{\frac{p}{2}} \right]\\
				&~\leq K_{p}{\mathbb{E}}\left[{\rm exp}\left(K_{p}\left(|\xi|+\int_{t}^{T}f_{s}ds \right)^{\frac{2}{\alpha^{*}}} \right)\right],~~~t\in[0,T].
			\end{aligned}
		\end{equation}
	\end{theorem}
	
	\begin{proof} The whole proof is divided into the following two steps.
		
		{\bf  Step 1.} We prove that inequalities (\ref{eq3.2}) and (\ref{eq3.3}) hold for any bounded solution $(Y_{t},Z_{t})_{t\in[0,T]}$ of BSDE $(\xi,g)$ when $|\xi|\leq \bar{c}$ and $|f_{t}|\leq \bar{c}e^{-t}$ hold for each $t\in[0,T]$ and some constant $\bar{c}>0$. The proof is mainly enlightened by Lemma 2.7 in Fan et al \cite{15}. Let $t\in[0,T]$ be fixed. Define the test function
		$$
		\psi(s,x):={\rm exp}\left(\mu(s)x^{\frac{2}{\alpha^{*}}}\right),~~~(s,x)\in[t,T]\times \left[x_{0},+\infty\right),
		$$
		where  $x_{0}> \left({\alpha^{*}}/{2}\right)^{{\alpha^{*}}/{2}}$ and $\mu(\cdot):[t,T]\rightarrow \mathbb{R}_{+}$ is a undetermined deterministic function satisfying $\mu(\cdot)\geq 1$ and $\mu^{\prime}(\cdot)>0$. Note that the larger the $\alpha^{*}$, the smaller the $\psi(s,x)$. A simple computation gives that for each $(s,x)\in[t,T]\times \left[x_{0},+\infty\right)$,
		$$
		\begin{aligned}
			&\psi_{s}(s,x)=\psi(s,x)\mu^{\prime}(s)x^{\frac{2}{\alpha^{*}}}>0\\
			&\psi_{x}(s,x)=\psi(s,x)\mu(s)\frac{2}{\alpha^{*}}x^{\frac{2}{\alpha^{*}}-1}>0\\
			&\psi_{xx}(s,x)=\psi(s,x)\mu(s)\frac{2}{\alpha^{*}}x^{\frac{2}{\alpha^{*}}-2}\left[\frac{2}{\alpha^{*}}\mu(s)x^{\frac{2}{\alpha^{*}}}-\left(1-\frac{2}{\alpha^{*}}\right)  \right]\\
			&~~~~~~~~~~~~\geq \psi(s,x)\mu(s)\frac{2}{\alpha^{*}}x^{\frac{2}{\alpha^{*}}-2}\left[\frac{2}{\alpha^{*}} x_{0}^{\frac{2}{\alpha^{*}}}-1\right]>0.
		\end{aligned}
		$$
	    For such a case, $x\mapsto \psi(s,x)$ is strictly convex on $\left[x_{0},+\infty\right)$. Note that $\lim\limits_{\alpha^{*}\rightarrow +\infty}\left({\alpha^{*}}/{2}\right)^{{\alpha^{*}}/{2}}=+\infty$. When $\alpha^{*}$ is getting larger, $x_{0}$ is getting larger. Define $A(s):=\int_{0}^{s}\beta(r)dr$ and
		\begin{equation}\label{eq3.4}
			\tag{3.4}
		\hat{Y}_{s}:=e^{A(s)}|Y_{s}|+k+\int_{t}^{s}e^{A(r)}f_{r}dr, ~~s\in[t,T],
	    \end{equation}
		where $k $ is a undetermined constant depending only on $\alpha$ bigger than $x_{0}$. Let $L(\cdot)$ denote the local time of $Y(\cdot)$ at 0. In view of assumption (EX1), applying It$\rm{\hat{o}}$-Tanaka$^{\prime}$s formula to $\psi(s,\hat{Y}_{s})$, we have
	    $$
		\begin{aligned}
			d\psi(s,\hat{Y}_{s})
			&=\left[\psi_{s}(s,\hat{Y}_{s})+\psi_{x}(s,\hat{Y}_{s})e^{A(s)}\left(-{\rm{sgn}}(Y_{s})g(s, Y_{s},Z_{s})+\beta(s)|Y_{s}|+f_{s} \right)\right.\\
			&~~~+\left.\frac{1}{2}\psi_{xx}(s,\hat{Y}_{s})e^{2A(s)}|Z_{s}|^{2}  \right]ds
			+\psi_{x}(s,\hat{Y}_{s})dL_{s}+\psi_{x}(s,\hat{Y}_{s})e^{A(s)}{\rm{sgn}}(Y_{s})Z_{s}\cdot dB_{s}\\
			&\geq \psi(s,\hat{Y}_{s}) \left[{\mu^{\prime}(s) {\hat{Y}_{s}}^{\frac{2}{\alpha^{*}}}-\frac{2}{\alpha^{*}}\mu(s){\hat{Y}_{s}}^{\frac{2}{\alpha^{*}}-1}e^{A(s)}\gamma(s)|Z_{s}|^{\alpha}}\right.\\ &~~~+\left.{\frac{1}{\alpha^{*}}\mu(s){\hat{Y}_{s}}^{\frac{2}{\alpha^{*}}-2}e^{2A(s)}\left(\frac{2}{\alpha^{*}}\mu(s){\hat{Y}_{s}}^{\frac{2}{\alpha^{*}}}
				-\left(1-\frac{2}{\alpha^{*}}\right)  \right)|Z_{s}|^{2}} \right]ds \\
			&~~~+\psi_{x}(s,\hat{Y}_{s})e^{A(s)}{\rm{sgn}}(Y_{s})Z_{s}\cdot dB_{s}, ~s\in[t,T].\\
		\end{aligned}
		$$
		For each $r\in[t,T]$ and each integer $m\geq 1$,  we define the following stopping time:
		$$
		\sigma_{m}^{r}:={\rm{inf}}\left\{s\in[r,T]: \int_{r}^{s}\left(\psi_{x}(\tau,\hat{Y}_{\tau}) \right)^{2}e^{2A(\tau)}|Z_{\tau}|^{2}d\tau\geq m \right\} \wedge T
		$$
		with the convention  ${\rm{inf}}~\emptyset=+\infty$. Then, by integrating on $[r,\sigma_{m}^{r}]$, we get that for each $m\geq 1$,
		$$
		\begin{aligned}
			\psi(r,\hat{Y}_{r})+X_{r}^{m} &\leq\psi(\sigma_{m}^{r},\hat{Y}_{\sigma_{m}^{r}})+\int_{r}^{\sigma_{m}^{r}}\psi(s,\hat{Y}_{s}) \left[{-\mu^{\prime}(s) {\hat{Y}_{s}}^{\frac{2}{\alpha^{*}}}+\frac{2}{\alpha^{*}}\mu(s){\hat{Y}_{s}}^{\frac{2}{\alpha^{*}}-1}e^{2A(s)}\gamma(s)|Z_{s}|^{\alpha}}\right.\\
			&~~~\left.{-\frac{1}{\alpha^{*}}\mu(s){\hat{Y}_{s}}^{\frac{2}{\alpha^{*}}-2}e^{2A(s)}\left(\frac{2}{\alpha^{*}}\mu(s){\hat{Y}_{s}}^{\frac{2}{\alpha^{*}}}-\left(1-\frac{2}{\alpha^{*}}\right)  \right)|Z_{s}|^{2}} \right]ds\\
			&\leq\psi(\sigma_{m}^{r},\hat{Y}_{\sigma_{m}^{r}})+\int_{r}^{\sigma_{m}^{r}}\psi(s,\hat{Y}_{s}) \left[{-\mu^{\prime}(s) {\hat{Y}_{s}}^{\frac{2}{\alpha^{*}}}+\frac{1}{\alpha^{*}}\mu(s){\hat{Y}_{s}}^{\frac{2}{\alpha^{*}}-2}}e^{2A(s)}\bigg(2\gamma(s)\hat{Y}_{s}|Z_{s}|^{\alpha}\bigg.\right.\\
			&~~~~~~~~~~~~~~~~~~~~~~~~~~~~~~~~~~~~~~~~~~\left.{\left.
				-\frac{2}{\alpha^{*}}\mu(s){\hat{Y}_{s}}^{\frac{2}{\alpha^{*}}}|Z_{s}|^{2}+  |Z_{s}|^{2}\right) } \right]ds,~~~r\in[t,T],
		\end{aligned}
		$$
		where
		$$
		X_{r}^{m}:=\int_{r}^{\sigma_{m}^{r}}\psi_{x}(s,\hat{Y}_{s})e^{A(s)}{\rm{sgn}}(Y_{s})Z_{s}\cdot dB_{s}.
		$$
		Using Young$^{\prime}$s inequality, we have
		$$
			\begin{aligned}
				2\gamma(s)\hat{Y}_{s}|Z_{s}|^{\alpha}=\left(\hat{Y}_{s}^{\frac{2}{\alpha^{*}}}|Z_{s}|^{2} \right)^{\frac{\alpha}{2}}\left(2\gamma(s)\hat{Y}_{s}^{1-\frac{\alpha}{\alpha^{*}}}\right)\leq \frac{1}{\alpha^{*}}{\hat{Y}_{s}}^{\frac{2}{\alpha^{*}}}|Z_{s}|^{2}+\widehat{K}\gamma^{\frac{2}{2-\alpha}}(s) \hat{Y}_{s}^{2}
			\end{aligned}
		$$
		with
		$$
		\widehat{K}:=(2-\alpha)\left(\frac{\alpha-1}{\alpha^{2}} \right)^{\frac{-\alpha}{2-\alpha}}>0.
		$$
		 Noticing $\mu(s)\geq 1$ and $\hat{Y}_{s}\geq k$, we choose $k:= \left((\alpha^{*})^{2}+\alpha^{*} \right)^{\frac{\alpha^{*}}{2}}>x_{0}$ such that ${\hat{Y}_{s}}^{\frac{2}{\alpha^{*}}}\geq 2{\rm{ln}}{\hat{Y}_{s}}$ and then
		$$
		\psi(s,\hat{Y}_{s}) {\hat{Y}_{s}}^{\frac{2}{\alpha^{*}}-2}\geq {\rm exp}\left(\hat{Y}_{s}^{\frac{2}{\alpha^{*}}}\right){\hat{Y}_{s}}^{\frac{2}{\alpha^{*}}-2}  \geq1.
		$$
		According to the last several inequalities,  we get that for each $m\geq 1$ and $r\in[t,T]$,
		\begin{equation}\label{eq3.5}
			\tag{3.5}
			\begin{aligned}
				&\psi(r,\hat{Y}_{r})+X_{r}^{m}\\
				&~\leq\psi(\sigma_{m}^{r},\hat{Y}_{\sigma_{m}^{r}})+\int_{r}^{\sigma_{m}^{r}}\left\{\psi(s,\hat{Y}_{s}) {\hat{Y}_{s}}^{\frac{2}{\alpha^{*}}}\left[{-\mu^{\prime}(s) +\frac{\widehat{K}}{\alpha^{*}}e^{2A(s)}\gamma^{\frac{2}{2-\alpha}}(s)\mu(s)}\right]\right.\\
				&~~~~~~~~~~~~~~~~~~~~~~~~~~~~~~~~~\left.-\frac{\mu(s)}{\alpha^{*}}\psi(s,\hat{Y}_{s}) {\hat{Y}_{s}}^{\frac{2}{\alpha^{*}}-2}e^{2A(s)}\left(
				\frac{1}{\alpha^{*}}{\hat{Y}_{s}}^{\frac{2}{\alpha^{*}}}-1  \right)|Z_{s}|^{2}  \right\}ds\\
				&~\leq \psi(\sigma_{m}^{r},\hat{Y}_{\sigma_{m}^{r}})+\int_{r}^{\sigma_{m}^{r}}\left\{\psi(s,\hat{Y}_{s}) {\hat{Y}_{s}}^{\frac{2}{\alpha^{*}}}\left[{-\mu^{\prime}(s) +\frac{\widehat{K}}{\alpha^{*}}e^{2A(s)}\gamma^{\frac{2}{2-\alpha}}(s)\mu(s)}\right]\right.\\
				&~~~~~~~~~~~~~~~~~~~~~~~~~~~~~~~~~-\left.
				\frac{1}{\alpha^{*}}\left(
				\frac{1}{\alpha^{*}}{k}^{\frac{2}{\alpha^{*}}}-1  \right)|Z_{s}|^{2}  \right\}ds,\\
				&~\leq \psi(\sigma_{m}^{r},\hat{Y}_{\sigma_{m}^{r}}) +\int_{r}^{\sigma_{m}^{r}}\left\{\psi(s,\hat{Y}_{s}) {\hat{Y}_{s}}^{\frac{2}{\alpha^{*}}}\left[{-\mu^{\prime}(s) +\frac{\widehat{K}}{\alpha^{*}}e^{2A(s)}\gamma^{\frac{2}{2-\alpha}}(s)\mu(s)}\right]-|Z_{s}|^{2}  \right\}ds.\\	
			\end{aligned}
		\end{equation}
		Now, we pick
		\begin{equation}\label{eq3.6}
			\tag{3.6}
			\mu(s):= \mu(0){\rm exp}\left({\frac{\widehat{K}}{\alpha^{*}}\int_{0}^{s}e^
				{2A(r)}\gamma^{\frac{2}{2-\alpha}}(r)dr }\right) ~~{\rm{and}}~~  \mu(0)\geq 1,~~~s\in[t,T].
		\end{equation}
		Then,  in view of ${\mathbb{E}}_{r}[X_{r}^{m}]=0$, by taking the conditional mathematical expectation with respect to $\mathcal{F}_{r}$ in $(\ref{eq3.5})$, we can obtain that for each $m\geq 1$ and $r\in[t,T]$,
		\begin{equation}\label{eq3.7}
			\psi(r,\hat{Y}_{r})	+{\mathbb{E}}_{r}\left[\int_{r}^{\sigma_{m}^{r}}|Z_{s}|^{2}ds \right]\leq {\mathbb{E}}_{r}\left[\psi(\sigma_{m}^{r},\hat{Y}_{\sigma_{m}^{r}}) \right].\tag{3.7}
		\end{equation}
		Furthermore, by Fatou$^{\prime}$s lemma and Lebesgue$^{\prime}$s dominated convergence theorem, as well as the boundedness of $Y_{\cdot}$ and $f_{\cdot}$, sending $m$ to infinity in $(\ref{eq3.7})$, we get
		\begin{equation}\label{eq3.8}
			\tag{3.8}
			\begin{aligned}
				&{\rm exp}\left(\mu(r)\left(e^{A(r)}|Y_{r}|+k+\int_{t}^{r}e^{A(s)}f_{s}ds \right)^{\frac{2}{\alpha^{*}}}\right)+{\mathbb{E}}_{r}\left[\int_{r}^{T}|Z_{s}|^{2}ds \right]\\
				&~\leq {\mathbb{E}}_{r}\left[{\rm exp}\left(\mu(T)\left(|\xi|e^{A(T)}+k+\int_{t}^{T}e^{A(s)}f_{s}ds \right)^{\frac{2}{\alpha^{*}}}\right) \right],~r\in[t,T].
			\end{aligned}
		\end{equation}
		Letting $r=t$ in the last inequality, we deduce that for each $t\in[0,T]$,
		\begin{equation}\label{eq3.9}
			\tag{3.9}
			\begin{aligned}
				&{\rm exp}\left(\mu(t)\left(e^{A(t)}|Y_{t}|+k \right)^{\frac{2}{\alpha^{*}}}\right)+{\mathbb{E}}_{t}\left[\int_{t}^{T}|Z_{s}|^{2}ds \right]\\
				&~\leq {\mathbb{E}}_{t}\left[{\rm exp}\left(\mu(T)\left(|\xi|e^{A(T)}+k+\int_{t}^{T}e^{A(s)}f_{s}ds \right)^{\frac{2}{\alpha^{*}}}\right) \right],
			\end{aligned}
		\end{equation}
		which yields that
		\begin{equation}\label{eq3.10}
			\tag{3.10}
			\begin{aligned}
				&~~~~{\rm exp}\left(|Y_{t}|^{\frac{2}{\alpha^{*}}}\right)+{\mathbb{E}}_{t}\left[\int_{t}^{T}|Z_{s}|^{2}ds \right]\leq K {\mathbb{E}}_{t}\left[{\rm exp}\left(K\left(|\xi|+\int_{t}^{T}f_{s}ds \right)^{\frac{2}{\alpha^{*}}}\right) \right]
			\end{aligned}
		\end{equation}
		with
		$$
		K:=\left({\rm{exp}}\left(\mu(T)k^{\frac{2}{\alpha^{*}}}\right)\right)\vee \left(\mu(T)e^{A(T)}\right).
		$$
		Thus, the desired inequality (\ref{eq3.2}) holds for any bounded solution $(Y_{t},Z_{t})_{t\in[0,T]}$ of BSDE $(\xi,g)$ when $|\xi|\leq \bar{c}$ and $|f_{t}|\leq \bar{c}e^{-t}$ hold for each $t\in[0,T]$ and some constant $\bar{c}>0$.
		
		Furthermore, in view of Doob$^{\prime}$s martingale inequality, inequality (\ref{eq3.8}) and $(a+b)^{\lambda}\leq a^{\lambda}+b^{\lambda}$ for $a,b\geq 0$ and $\lambda\in(0,1)$, we deduce that for each $p>1$ and $t\in[0,T]$,
		\begin{equation}\label{eq3.11}
			\tag{3.11}
			\begin{aligned}
				&{\mathbb{E}}\left[{\rm exp}\left(p\left(\sup\limits_{r\in[t,T]}|Y_{r}|\right)^{\frac{2}{\alpha^{*}}}\right)\right]
				\leq{\mathbb{E}}\left[\sup\limits_{r\in[t,T]}{\rm exp}\left(p\mu(r)\left(e^{A(r)}|Y_{r}|+k+\int_{t}^{r}e^{A(s)}f_{s}ds \right)^{\frac{2}{\alpha^{*}}} \right)\right]\\
				&~\leq {\mathbb{E}}\left[\sup\limits_{r\in[t,T]}\left\{{\mathbb{E}}_{r}\left[{\rm exp}\left(\mu(T)\left(|\xi|e^{A(T)}+k+\int_{t}^{T}e^{A(s)}f_{s}ds \right)^{\frac{2}{\alpha^{*}}}\right) \right]\right\}^{p}\right]\\
				&~\leq \left(\frac{p}{p-1}\right)^{p}{\mathbb{E}}\left[{\rm exp}\left(p\mu(T)\left(|\xi|e^{A(T)}+k+\int_{t}^{T}e^{A(s)}f_{s}ds \right)^{\frac{2}{\alpha^{*}}}\right) \right]\\
				&~\leq \left(\frac{p}{p-1}\right)^{p}e^{p\mu(T)k^{\frac{2}{\alpha^{*}}}}{\mathbb{E}}\left[{\rm exp}\left(p\mu(T)e^{A(T)}\left(|\xi|+\int_{t}^{T}f_{s}ds \right)^{\frac{2}{\alpha^{*}}}\right) \right]<+\infty.
			\end{aligned}
		\end{equation}
		On the other hand, using inequalities $(\ref{eq3.5})$ and $(\ref{eq3.6})$,  we can deduce that for each $m\geq 1$ and $p>1$,
		$$
		\begin{aligned}
			{\mathbb{E}}\left[\left(\int_{r}^{\sigma_{m}^{r}}|Z_{s}|^{2}ds \right)^{\frac{p}{2}} \right]&\leq {\mathbb{E}}\left[\left(\psi(\sigma_{m}^{r},\hat{Y}_{\sigma_{m}^{r}})+|X_{r}^{m}| \right)^{\frac{p}{2}} \right]\\
			&\leq 2^{p-1}\left\{{\mathbb{E}}\left[\left(\psi(\sigma_{m}^{r},\hat{Y}_{\sigma_{m}^{r}}) \right)^{\frac{p}{2}}\right]+{\mathbb{E}}\left[|X_{r}^{m}|^{\frac{p}{2}} \right] \right\},~~r\in[t,T].
		\end{aligned}
		$$
		In view of the definition of $X_{r}^{m}$ and the fact that $\hat{Y}_{t}\geq k \geq 1$, from Burkholder-Davis-Gundy$^{\prime}$s inequality and H$\ddot{\rm{o}}$lder$^{\prime}$s inequality, we observe that
		$$
		\begin{aligned}
			2^{p-1}	{\mathbb{E}}\left[ |X_{r}^{m}|^{\frac{p}{2}}\right]&\leq 2^{p-1} (\mu(T))^{\frac{p}{2}}e^{\frac{p}{2}A(T)}{\mathbb{E}}\left[\left(\int_{r}^{\sigma_{m}^{r}}\left(\psi(s,\hat{Y}_{s})\right)^{2}|Z_{s}|^{2}ds \right)^{\frac{p}{4}} \right]\\
			&\leq 2^{p-1} (\mu(T))^{\frac{p}{2}}e^{\frac{p}{2}A(T)}{\mathbb{E}}\left[\sup\limits_{s\in[r,\sigma_{m}^{r}]}\left(\psi(s,\hat{Y}_{s})\right)^{\frac{p}{2}} \left(\int_{r}^{\sigma_{m}^{r}}|Z_{s}|^{2}ds \right)^{\frac{p}{4}} \right]\\
			&\leq 2^{2p-3} (\mu(T))^{p}e^{pA(T)}{\mathbb{E}}\left[\sup\limits_{s\in[r,\sigma_{m}^{r}]}\left(\psi(s,\hat{Y}_{s})\right)^{p}\right] +\frac{1}{2}{\mathbb{E}}\left[\left(\int_{r}^{\sigma_{m}^{r}}|Z_{s}|^{2}ds \right)^{\frac{p}{2}} \right].
		\end{aligned}
		$$
		Then,
		$$
		\begin{aligned}
			{\mathbb{E}}\left[\left(\int_{r}^{\sigma_{m}^{r}}|Z_{s}|^{2}ds \right)^{\frac{p}{2}} \right]
			&\leq 2^{p}{\mathbb{E}}\left[\left(\psi(\sigma_{m}^{r},\hat{Y}_{\sigma_{m}^{r}}) \right)^{\frac{p}{2}}\right]+ 2^{2p-2} (\mu(T))^{p}e^{pA(T)}{\mathbb{E}}\left[\sup\limits_{s\in[r,\sigma_{m}^{r}]}\left(\psi(s,\hat{Y}_{s})\right)^{p}\right]\\
			&\leq(8\mu(T))^{p}e^{pA(T)}{\mathbb{E}}\left[\sup\limits_{s\in[r,\sigma_{m}^{r}]}\left(\psi(s,\hat{Y}_{s})\right)^{p}\right],~r\in[t,T].
		\end{aligned}
		$$
		From the last inequality, sending $m$ to infinity and $r=t$, and using Fatou$^{\prime}$s lemma, Lebesgue$^{\prime}$s dominated  convergence theorem and inequality (\ref{eq3.11}), we get that for each $t\in [0,T]$,
		\begin{equation}\label{eq3.12}
			\tag{3.12}
			\begin{aligned}
				&{\mathbb{E}}\left[\left(\int_{t}^{T}|Z_{s}|^{2}ds \right)^{\frac{p}{2}} \right]\leq(8\mu(T))^{p}e^{pA(T)}{\mathbb{E}}\left[\sup\limits_{s\in[t,T]}\left(\psi(s,\hat{Y}_{s})\right)^{p}\right]\\
				&~=(8\mu(T))^{p}e^{pA(T)}{\mathbb{E}}\left[\sup\limits_{r\in[t,T]}{\rm exp}\left(p\mu(r)\left(e^{A(r)}|Y_{r}|+k+\int_{t}^{r}e^{A(s)}f_{s}ds \right)^{\frac{2}{\alpha^{*}}} \right)\right]\\
				&~\leq (8\mu(T))^{p}e^{pA(T)}\left(\frac{p}{p-1}\right)^{p}e^{p\mu(T)k^{\frac{2}{\alpha^{*}}}}{\mathbb{E}}\left[{\rm exp}\left(p\mu(T)e^{A(T)}\left(|\xi|+\int_{t}^{T}f_{s}ds \right)^{\frac{2}{\alpha^{*}}}\right) \right].
			\end{aligned}
		\end{equation}
		Combining inequalities $(\ref{eq3.11})$ and $(\ref{eq3.12})$, for each $p> 1$, $t\in[0,T]$, we have
		\begin{equation}\label{eq3.13}
			\tag{3.13}
			\begin{aligned}
				&{\mathbb{E}}\left[{\rm exp}\left(p\left(\sup\limits_{s\in[t,T]}|Y_{s}|\right)^{\frac{2}{\alpha^{*}}}\right)\right] +{\mathbb{E}}\left[\left(\int_{t}^{T}|Z_{s}|^{2}ds \right)^{\frac{p}{2}} \right]\\
				&~\leq K_{p}{\mathbb{E}}\left[{\rm exp}\left(K_{p}\left(|\xi|+\int_{t}^{T}f_{s}ds \right)^{\frac{2}{\alpha^{*}}} \right)\right],
			\end{aligned}
		\end{equation}
		where
		$$
		K_{p}:=\left(\left(\frac{p}{p-1} \right)^{p}\left((8\mu(T))^{p}e^{pA(T)}+1 \right)e^{p\mu(T)k^{\frac{2}{\alpha^{*}}}}\right)\vee \left(p\mu(T)e^{A(T)}\right).\vspace{0.1cm}
		$$
		Thus, the desired inequality (\ref{eq3.3})  holds for any bounded solution $(Y_{t},Z_{t})_{t\in[0,T]}$ of BSDE $(\xi,g)$ when $|\xi|\leq \bar{c}$ and $|f_{t}|\leq \bar{c}e^{-t}$ hold for each $t\in[0,T]$ and some constant $\bar{c}>0$.\vspace{0.2cm}
		
		{\bf  Step 2.} Based on Step 1, we use the localization procedure of Briand and Hu \cite{2} to construct the desired solution.  For each pair of given positive integers $n,~q\geq1$ and $(\omega,t,y,z)\in\Omega\times [0,T]\times\mathbb{R}^{1+d}$, let
		$$
		\xi^{n,q}:=\xi^{+}\wedge n-\xi^{-}\wedge q~~~{\rm{and}}~~~g^{n,q}(\omega,t,y,z):=g^{+}(\omega,t,y,z)\wedge (ne^{-t})-g^{-}(\omega,t,y,z)\wedge (qe^{-t}).
		$$
		It is immediately seen that $|\xi^{n,q}|\leq |\xi|\wedge(n\vee q)$ and $|g^{n,q}|\leq|g|\wedge[(n\vee q)e^{-t}]$. It can also be verified that the generator $g^{n,q}$ satisfies assumptions (H1), (H2), (EX1) and (EX2) with $u(t)=(n\vee q)e^{-t}$ and $f_{t}\wedge[(n\vee q)e^{-t}]$ instead of $f_{t}$. Then,
		thanks to Lemma \ref{Le:2.1}, the following BSDE $(\xi^{n,q},g^{n,q})$ has  a minimal bounded solution $(Y_{t}^{n,q},Z_{t}^{n,q})_{t\in[0,T]}$ in the space of $\mathcal{S}^{\infty}(0,T;\mathbb{R})\times \mathcal{M}^{2}(0,T;\mathbb{R}^{d})$:
		$$
		Y_{t}^{n,q}=\xi^{n,q} +\int_{t}^{T}g^{n,q}(s,Y_{s}^{n,q},Z_{s}^{n,q})ds-\int_{t}^{T}Z_{s}^{n,q}\cdot dB_{s},~~~~~t\in [0,T].
		$$
		From the comparison result stated in Lemma \ref{Le:2.1}, we know that $(Y_{t}^{n,q})_{t\in[0,T]}$ is non-decreasing in $n$ and non-increasing in $q$. Furthermore, by virtue of the assertion $(\ref{eq3.10})$ of Step 1,  for each $t\in[0,T]$ and $n,q\geq 1$, we have
		$$
		\begin{aligned}
			{\rm exp}\left(|Y_{t}^{n,q}| ^{\frac{2}{\alpha^{*}}}\right)&\leq K {\mathbb{E}}_{t}\left[{\rm exp}\left(K\left(|\xi^{n,q}|+\int_{t}^{T}f_{s}\wedge [(n\vee q)e^{-s}]ds \right)^{\frac{2}{\alpha^{*}}}\right) \right]\\
			&\leq K {\mathbb{E}}_{t}\left[{\rm exp}\left(K\left(|\xi|+\int_{0}^{T}f_{s}ds \right)^{\frac{2}{\alpha^{*}}}\right) \right],
		\end{aligned}
		$$
		which implies
		$$
			|Y_{t}^{n,q}|\leq \left( {\rm{ln}}K+{\rm{ln}}\left\{ {\mathbb{E}}_{t}\left[{\rm exp}\left(K\left(|\xi|+\int_{0}^{T}f_{s}ds \right)^{\frac{2}{\alpha^{*}}}\right) \right]\right\} \right)^{\frac{\alpha^{*}}{2}}.
		$$
		Thus, in view of the last inequality and the assumption (EX2), we can  apply the localization procedure developed initially in Briand and Hu \cite{2} to obtain an $(\mathcal{F}_{t})$-progressively measurable $\mathbb{R}^{d}$-valued process $(Z_{t})_{t\in[0,T]}$  such that $Z_{\cdot}^{n,q}$ tends to $Z_{\cdot}$ as $n,q$ tend to infinity and the pair of $(Y_{\cdot}:=\inf\limits_{q}\sup\limits_{n}Y_{\cdot}^{n,q},Z_{\cdot})$ is a solution to BSDE $(\xi,g)$. And, the desired inequalities (\ref{eq3.2}) and (\ref{eq3.3}) for $(Y_{\cdot},Z_{\cdot})$ follow immediately from Step 1 and the definition of $(Y_{\cdot},Z_{\cdot})$. The proof is then complete.
	\end{proof}

	\begin{remark}\label{Re:3.2}
		From the above proof, it is easy to verify that in the assumptions of Theorem \ref{The:3.2}, if $|\xi|$ is replaced with $\xi^{+}$ and (EX1) is replaced by the following  (EX1$^{\prime}$):

\begin{enumerate}
		\renewcommand{\theenumi}{($\rm{EX1}^{\prime}$\alph{enumii})}
		\renewcommand{\labelenumi}{\theenumi}
		\item\label{UN2} d$\mathbb{P}\times$ d$t-a.e.,$ for each $(y,z)\in \mathbb{R}\times \mathbb{R}^{d}$,
     	$$
	  {{\bf 1}}_{\left\{y>0\right\}}g(\omega,t,y,z)\leq f_{t}(\omega)+\beta(t)|y|+\gamma(t)|z|^{\alpha},
      	$$	
	\end{enumerate}
		then for any solution $(Y_{t},Z_{t})_{t\in[0,T]}$ of BSDE $(\xi,g)$ such that $\sup\limits_{t\in[0,T]}Y_{t}^{+}$ has sub-exponential moments of any order, there exists a constant $k\geq 1$ and a differentiable function $\mu(\cdot):[0,T]\mapsto \left[1,+\infty\right)$  such that $\mathbb{P}-a.s.$, for each $t\in[0,T]$,
		$$
		\begin{aligned}
			&{\rm exp}\left(\mu(t)\left(e^{A(t)}Y_{t}^{+}+k \right)^{\frac{2}{\alpha^{*}}}\right)+{\mathbb{E}}_{t}\left[\int_{t}^{T}{\bf{1}}_{\left\{Y_{s}>0\right\}}|Z_{s}|^{2}ds \right]\\
			&~\leq {\mathbb{E}}_{t}\left[{\rm exp}\left(\mu(T)\left(\xi^{+}e^{A(T)}+k+\int_{t}^{T}e^{A(s)}f_{s}ds \right)^{\frac{2}{\alpha^{*}}}\right) \right].
		\end{aligned}
		$$
		Consequently, there exists a constant $K>0$ depending only on $(\alpha,T,\beta(\cdot),\gamma(\cdot))$ such that $\mathbb{P}-a.s.$, for each $t\in[0,T]$,
		$$
			\begin{aligned}
				{\rm exp}\left(\left(Y_{t}^{+} \right)^{\frac{2}{\alpha^{*}}}\right)+{\mathbb{E}}_{t}\left[\int_{t}^{T}{\bf{1}}_{\left\{Y_{s}>0\right\}}|Z_{s}|^{2}ds \right]
				\leq K {\mathbb{E}}_{t}\left[{\rm exp}\left(K\left(\xi^{+}+\int_{t}^{T}f_{s}ds \right)^{\frac{2}{\alpha^{*}}}\right) \right].
			\end{aligned}
		$$
		In the above proof one only  needs to  use $Y_{\cdot}^{+},~{\bf 1}_{Y_{\cdot}>0}Y_{\cdot}$ and $\frac{1}{2}L_{\cdot}$ instead of $|Y_{\cdot}|,~{\rm{sgn}}(Y_{\cdot})$ and $L_{\cdot}$, respectively.
	\end{remark}

	\section{Uniqueness of the  unbounded solution}\label{Title:4}
	In this section, we will establish a general comparison theorem and a general uniqueness and existence theorem for unbounded solutions of BSDEs, which extend the corresponding results in Fan and Hu \cite{12}.
	
	\subsection{Comparison theorem}\label{Title:4.1}
	 Let us introduce the following time-varying extended convexity assumption on the  generator $g$.
	
	\begin{enumerate}
		\renewcommand{\theenumi}{($\rm{UN}$)}
		\renewcommand{\labelenumi}{\theenumi}
		\item\label{UN} $g$ satisfies either of the following two conditions:
		
		$(i)$~d$\mathbb{P}\times$ d$t-a.e.$, for each $(y_{i},z_{i})\in  \mathbb{R}^{1+d},~i=1,2$ and  $\theta \in(0,1)$,
		\begin{equation}\label{eq4.1}
			\tag{4.1}
			\begin{aligned}
				&{{\bf 1}}_{\left\{y_{1}-\theta y_{2}>0\right\}}(g(\omega,t,y_{1},z_{1})-\theta g(\omega,t,y_{2},z_{2}))\\
				&~\leq (1-\theta)(f_{t}(\omega)+\beta(t)|y_{2}|+\gamma(t)\left[{\rm{ln}}\left(e+|z_{2}|\right)\right]^{\frac{\alpha^{*}}{2}}+\beta(t)|\delta_{\theta}y|+\gamma(t)|\delta_{\theta}z|^{\alpha});\\	
			\end{aligned}
		\end{equation}	
		
		$(ii)$~d$\mathbb{P}\times$ d$t-a.e.$, for each $(y_{i},z_{i})\in  \mathbb{R}^{1+d},~i=1,2$ and $\theta\in(0,1)$,
		\begin{equation}\label{eq4.2}
			\tag{4.2}
			\begin{aligned}
				&-{{\bf 1}}_{\left\{y_{1}-\theta y_{2}<0\right\}}(g(\omega,t,y_{1},z_{1})-\theta g(\omega,t,y_{2},z_{2}))\\
				&~\leq (1-\theta)(f_{t}(\omega)+\beta(t)|y_{2}|+\gamma(t)\left[{\rm{ln}}\left(e+|z_{2}|\right)\right]^{\frac{\alpha^{*}}{2}}+\beta(t)|\delta_{\theta}y|+\gamma(t)|\delta_{\theta}z|^{\alpha}),\\
			\end{aligned}
		\end{equation}
		where
		$$\delta_{\theta}y:=\frac{y_{1}-\theta y_{2}}{1-\theta},~~~
		\delta_{\theta}z:=\frac{z_{1}-\theta z_{2}}{1-\theta}.\vspace{0.1cm}
		$$	
		Assume further that $\gamma(\cdot):[0,T]\mapsto \mathbb{R}_{+}$ satisfies
		\begin{equation}\label{eq4.3}
			\tag{4.3}
			0<\int_{0}^{T}\gamma(t)dt<+\infty.
		\end{equation}

	\end{enumerate}

	\begin{remark}\label{Re:4.1}
		$(i)$ It is clear that  even though for the case of $T<+\infty$ and both $\beta(\cdot)$ and $\gamma(\cdot)$ are nonnegative constants, due to presence of the term $\gamma(t)\left[{\rm{ln}}(e+|z_{2}|)\right]^{\frac{\alpha^{*}}{2}}$, the assumption (UN) is strictly weaker than  assumption $({\rm{H}}2^{\prime})$ in Fan and Hu \cite{12}. And, it is easy to verify that if $g(\omega,t,y,z)$ satisfies (i) of (UN), then $\hat{g}(\omega,t,y,z):=-g(\omega,t,-y,-z)$ satisfies (ii) of (UN).
		
		$(ii)$ Using a similar argument as in  page 29 of Fan et al \cite{15}, we can conclude that if $g$ satisfies assumption (UN), then d$\mathbb{P}\times$ d$t-a.e.$, $g(\omega,t,\cdot,\cdot)$ is locally Lipschitz  continuous on $\mathbb{R}\times \mathbb{R}^{d}$.
		
		$(iii)$ It is easy to verify that if the generator $g_{i}$ satisfies inequality (\ref{eq4.1}) (resp. (\ref{eq4.2})) for $i=1,2$, then for each $k_{i}>0$,  $i=1,2$, all of $k_{1}g_{1}+k_{2}g_{2},~g_{1}\vee g_{2}$ and $g_{1}\wedge g_{2}$  satisfy inequality (\ref{eq4.1}) (resp. (\ref{eq4.2})).
		
		$(iv)$ Letting $y_{1}=y_{2}=y$ and $z_{1}=z_{2}=z$ in $(\ref{eq4.1})$ and $(\ref{eq4.2})$ respectively yields that for some constant $k>0$ depending only on $\alpha$,
		$$
		{\bf 1}_{\left\{y>0\right\}}g(\omega,t,y,z)\leq f_{t}(\omega)+2\beta(t)|y|+k\gamma(t)|z|^{\alpha}
		$$
		and
		$$
		-{\bf 1}_{\left\{y<0\right\}}g(\omega,t,y,z)\leq f_{t}(\omega)+2\beta(t)|y|+k\gamma(t)|z|^{\alpha},
		$$
		which yield that the generator $g$ has a time-varying one-sided  linear growth  in   $y$ and a time-varying sub-quadratic growth   in  $z$.
		
		$(v)$ Letting  $z_{1}=z_{2}=z$ in $(\ref{eq4.1})$ or $(\ref{eq4.2})$ and then letting $\theta \rightarrow 1$ yields that	
		$$
		{\bf 1}_{\left\{y_{1}-y_{2}>0\right\}}(g(\omega,t,y_{1},z)-g(\omega,t,y_{2},z))\leq \beta(t)|y_{1}-y_{2}|,
		$$
		which means that the generator $g$ satisfies a time-varying monotonicity condition  in  $y$.
	\end{remark}

	The following Theorem \ref{The:4.2} establishes a general comparison theorem for unbounded solutions of BSDEs under assumption $(\rm{UN})$, which is the main result of this subsection.
	
	\begin{theorem}\label{The:4.2}
		Assume that $\xi$ and $\xi^{\prime}$ are two terminal values satisfying $\mathbb{P}-a.s.$, $~\xi\leq \xi^{\prime}$, $g$ and $g^{\prime}$ are two  generators,  and $(Y_{t},Z_{t})_{t\in[0,T]}$ and $(Y_{t}^{\prime}, Z_{t}^{\prime})_{t\in[0,T]}$ are, respectively,  solutions to BSDE $(\xi,g)$ and BSDE $(\xi^{\prime},g^{\prime})$ such that for each $p>1$,
		\begin{equation}\label{eq4.4}
			\tag{4.4}
			 {\mathbb{E}}\left[{\rm{exp}}\left(p\left(\sup\limits_{t\in[0,T]}(|Y_{t}|+|Y_{t}^{\prime}|)\right)^{\frac{2}{\alpha^{*}}}\right)  \right] <+\infty,
		\end{equation}
		and $Z_{\cdot},Z_{\cdot}^{\prime}\in\mathcal{M}^{p}(0,T;\mathbb{R}^{d})$.   If $g~({\rm{resp.}}~ g^{\prime})$ satisfies assumption  (UN) and
		\begin{equation}\label{eq4.5}
			\tag{4.5}
			{\rm{d}}\mathbb{P}\times {\rm{d}}t-a.e.,~g(t,Y_{t}^{\prime},Z_{t}^{\prime})\leq g^{\prime}(t,Y_{t}^{\prime},Z_{t}^{\prime})~~({\rm{resp.}}~ g(t,Y_{t},Z_{t})\leq g^{\prime}(t,Y_{t},Z_{t})),
		\end{equation}
		then $\mathbb{P}-a.s.$, for each $t\in[0,T]$, $Y_{t}\leq Y_{t}^{\prime}$.
	\end{theorem}
	
	\begin{proof}
		The proof is divided into the following four steps.
		
		{\bf Step 1.}  Define
		$$
		\hat{f_{t}}:=f_{t}+\beta(t)|Y_{t}^{\prime}|+\gamma(t)\left[{\rm{ln}}\left(e+|Z^{\prime}_{t}|\right)\right]^{\frac{\alpha^{*}}{2}},~~t\in[0,T].
		$$
		We first prove that $\int_{0}^{T}\hat{f}_{t}dt$ has sub-exponential moments of any order, i.e.,
		\begin{equation}\label{eq4.6}
			{\mathbb{E}}\left[{\rm exp}\left(p\left(\int_{0}^{T}\hat{f}_{s}ds \right)^{\frac{2}{\alpha^{*}}} \right)\right]<+\infty,~~~\forall~ p>1. \tag{4.6}
		\end{equation}
		In fact,  by virtue of $\beta(\cdot)\in L^{1}([0,T];\mathbb{R}_{+})$ and inequality (\ref{eq4.4}), we have for each $p>1$,
		\begin{equation}\label{eq4.7}
			\tag{4.7}
			\begin{aligned}
				{\mathbb{E}}\left[{\rm exp}\left(p\left(\int_{0}^{T}\beta(t)|Y_{t}^{\prime}|dt \right)^{\frac{2}{\alpha^{*}}}\right)\right]&\leq {\mathbb{E}}\left[{\rm exp}\left(p\left(\sup\limits_{t\in[0,T]}|Y_{t}^{\prime}|\int_{0}^{T}\beta(t)dt \right)^{\frac{2}{\alpha^{*}}}\right)\right]\\
                &= {\mathbb{E}}\left[{\rm exp}\left(p\left(\int_{0}^{T}\beta(t)dt\right)^{\frac{2}{\alpha^{*}}}\left(\sup\limits_{t\in[0,T]}|Y_{t}^{\prime}| \right)^{\frac{2}{\alpha^{*}}}\right)\right]<+\infty.
			\end{aligned}
		\end{equation}
		On the other hand,
		by combining inequality (\ref{eq4.3}), H\"{o}lder$^{\prime}$s inequality  and the basic inequality $(a+b)^{q}\leq 2^{q}(a^{q}+b^{q})$ for $a,b>0$ and $q>0$,  we deduce that for each $p>1$,
		$$
		\begin{aligned}
			&{\mathbb{E}}\left[\left(\int_{0}^{T}\gamma(t)|Z_{t}^{\prime}|dt\right)^{\delta_{p}}\right]\leq {\mathbb{E}}\left[\left(\int_{0}^{T}\gamma(t)(1+|Z_{t}^{\prime}|^{\alpha})dt\right)^{\delta_{p}}\right]\\
			&~\leq 2^{\delta_{p}}\left(\int_{0}^{T}\gamma(t)dt\right)^{\delta_{p}}+2^{\delta_{p}} {\mathbb{E}}\left[\left(\int_{0}^{T}\gamma^{\frac{2}{2-\alpha}}(t)dt\right)^{\frac{(2-\alpha)}{2}\delta_{p}}\left(\int_{0}^{T}|Z_{t}^{\prime}|^{2}dt\right)^{\frac{\alpha}{2}\delta_{p}}\right]<+\infty,
		\end{aligned}
		$$
    	where
    	$$\delta_{p}:=p\left(\int_{0}^{T}\gamma(t)dt\right)^{\frac{2}{\alpha^{*}}}.$$
    	Note that $d\mu_{t}:=\frac{\gamma(t)dt}{\int_{0}^{T}\gamma(t)dt}$ defines a probability measure on $[0,T]$ and ${[{\rm{ln}}(k_{\alpha}+x)]^{\frac{\alpha^{*}}{2}}}$ with $k_{\alpha}:={\rm{exp}}\left(\frac{\alpha^{*}}{2}\right)$ is a concave function on $\mathbb{R}_{+}$. It follows from Jensen$^{\prime}$s inequality and the last inequality that
		\begin{equation}\label{eq4.8}
			\tag{4.8}
			\begin{aligned}
				&{\mathbb{E}}\left[{\rm exp}\left(p\left(\int_{0}^{T}\gamma(t)\left({\rm{ln}}(e+|Z_{t}^{\prime}|)\right)^\frac{\alpha^{*}}{2}dt \right)^{\frac{2}{\alpha^{*}}}\right)\right]\\
				&~\leq {\mathbb{E}}\left[{\rm exp}\left(\delta_{p}\left(\int_{0}^{T}\left({\rm{ln}}(k_{\alpha}+|Z_{t}^{\prime}|)\right)^\frac{\alpha^{*}}{2}d\mu_{t}  \right)^{\frac{2}{\alpha^{*}}}\right)\right]\\
				&~\leq {\mathbb{E}}\left[{\rm exp} \left(\delta_{p} {\rm{ln}}\left(k_{\alpha}+ {\int_{0}^{T}|Z_{t}^{\prime}|d\mu_{t}}\right) \right)  \right]= {\mathbb{E}}\left[\left(k_{\alpha}+ {\int_{0}^{T}|Z_{t}^{\prime}|d\mu_{t}}\right)^{\delta_{p}}\right]\\
				&~\leq 2^{\delta_{p}}\left(k_{\alpha}^{\delta_{p}}+{\mathbb{E}}\left[\left(\int_{0}^{T}|Z_{t}^{\prime}|d\mu_{t}\right)^{\delta_{p}}\right]\right)\\
				&~= 2^{\delta_{p}}\left(k_{\alpha}^{\delta_{p}}+\frac{1}{\left(\int_{0}^{T}\gamma(t)dt\right)^{\delta_{p}}}{\mathbb{E}}\left[\left(\int_{0}^{T}\gamma(t)|Z_{t}^{\prime}|dt\right)^{\delta_{p}}\right]\right)<+\infty.
			\end{aligned}
		\end{equation}
  In view of inequalities  (\ref{eq4.7}) and (\ref{eq4.8}) together with integrability of process $f_{\cdot}$, the desired assertion (\ref{eq4.6}) follows immediately from H\"{o}lder$^{\prime}$s inequality.
		
		{\bf Step 2.}
		Assume that the generator $g$ satisfies  $(i)$ of assumption (UN), and ${\rm{d}}\mathbb{P}\times {\rm{d}}t-a.e.$, $g(t,Y_{t}^{\prime},Z_{t}^{\prime})\leq g^{\prime}(t,Y_{t}^{\prime},Z_{t}^{\prime})$. We proceed by using the $\theta$-difference technique developed in Briand and Hu \cite{3}. For each fixed $\theta\in(0,1)$, define
		$$
		\delta_{\theta}U_{\cdot}:=\frac{Y_{\cdot}-\theta Y_{\cdot}^{\prime}}{1-\theta}~~~{\rm{and}}~~~\delta_{\theta}V_{\cdot}:=\frac{Z_{\cdot}-\theta Z_{\cdot}^{\prime}}{1-\theta}.
		$$
		Then the pair $(\delta_{\theta}U_{\cdot},\delta_{\theta}V_{\cdot})$ satisfies the following BSDE:
		$$
		\delta_{\theta}U_{t}=\delta_{\theta}U_{T}+\int_{t}^{T}\delta_{\theta}g(s,\delta_{\theta}U_{s},\delta_{\theta}V_{s})ds-\int_{t}^{T}\delta_{\theta}V_{s}\cdot dB_{s},~~~t\in[0,T],
		$$
		where ${\rm{d}}\mathbb{P}\times {\rm{d}}s-a.e.$, for each $(y,z)\in \mathbb{R}\times \mathbb{R}^{d}$,
		\begin{equation}\label{eq4.9}
			\tag{4.9}
			\begin{aligned}
				\delta_{\theta}g(s,y,z):=&\frac{1}{1-\theta}\left[g(s,(1-\theta)y+\theta Y_{s}^{\prime},(1-\theta)z+\theta Z_{s}^{\prime})-\theta g(s,Y_{s}^{\prime},Z_{s}^{\prime})\right]\\
				&+\frac{\theta}{1-\theta}[g(s,Y_{s}^{\prime},Z_{s}^{\prime})-g^{\prime}(s,Y_{s}^{\prime},Z_{s}^{\prime})].
			\end{aligned}
		\end{equation}
		It follows from the assumption (UN) and inequality (\ref{eq4.9}) that ${\rm{d}}\mathbb{P}\times {\rm{d}}s-a.e.$, for each $(y,z)\in \mathbb{R}\times \mathbb{R}^{d}$,
		\begin{equation}\label{eq4.10}
			\tag{4.10}
			\begin{aligned}
				{{\bf 1}}_{\left\{y>0\right\}}\delta_{\theta}g(s,y,z)&\leq f_{s}+\beta(s)|Y_{s}^{\prime}|+\gamma(s)\left[{\rm{ln}}\left(e+|Z^{\prime}_{s}|\right)\right]^{\frac{\alpha^{*}}{2}}+\beta(s)|y|+\gamma(s)|z|^{\alpha}\\
				&=\hat{f}_{s}+\beta(s)|y|+\gamma(s)|z|^{\alpha},	
			\end{aligned}
		\end{equation}
		which together with inequality (\ref{eq4.6}) means that the generator $\delta_{\theta}g$ satisfies assumption $(\rm{EX}1^{\prime})$ with $\hat{f}_{\cdot}$ instead of $f_{\cdot}$, and
		$$
		\delta_{\theta}U_{T}^{+}=\frac{(\xi-\theta\xi^{\prime})^{+}}{1-\theta}
		=\frac{[\xi-\theta\xi+\theta(\xi-\xi^{\prime})]^{+}}{1-\theta}\leq\xi^{+}.
		$$
	    Then, according to Remark {\ref{Re:3.2}}, we know that there exists a constant $K>0$ depending only on $(\alpha, T, \beta(\cdot),\gamma(\cdot))$  such that for each $t\in[0,T]$,
		$$
		{\rm exp}\left(\delta_{\theta} U_{t}^{+} \right)^{\frac{2}{\alpha^{*}}}\leq K {\mathbb{E}}_{t}\left[{\rm exp}\left(K\left(\xi^{+}+\int_{0}^{T}f_{s}ds \right)^{\frac{2}{\alpha^{*}}}\right) \right],
		$$
		and then
		$$
		(Y_{t}-\theta Y_{t}^{\prime})^{+}
		\leq (1-\theta)\left\{{\rm{ln}}\left[K {\mathbb{E}}_{t}\left[{\rm exp}\left(K\left(\xi^{+}+\int_{0}^{T}\hat{f_{s}}ds \right)^{\frac{2}{\alpha^{*}}}\right) \right]\right]\right\}^{\frac{\alpha^{*}}{2}}.
		$$
		Consequently, the desired conclusion follows by sending $\theta \rightarrow 1$ in the last inequality. That is to say, $\mathbb{P}-a.s.$, for each $t\in[0,T]$, $Y_{t}\leq Y_{t}^{\prime}$.
		
		{\bf Step 3.} We discuss the case that the generator $g^{\prime}$ satisfies  $(i)$ of assumption (UN), and ${\rm{d}}\mathbb{P}\times {\rm{d}}t-a.e.$, $g(t,Y_{t},Z_{t})\leq g^{\prime}(t,Y_{t},Z_{t})$. In this case the generator $\delta_{\theta}g$ in  $(\ref{eq4.9})$ should be replaced by the following equality
		$$
		\begin{aligned}
			\delta_{\theta}g(s,y,z):=&\frac{1}{1-\theta}\left[g(s, Y_{s}, Z_{s})-g^{\prime}(s, Y_{s}, Z_{s})\right]\\
			&+\frac{1}{1-\theta}[g^{\prime}(s,(1-\theta)y+\theta Y_{s}^{\prime},(1-\theta)z+\theta Z_{s}^{\prime})-\theta g^{\prime}(s, Y_{s}^{\prime}, Z_{s}^{\prime})].
		\end{aligned}
		$$
		It is easy to verify that $\delta_{\theta}g$ satisfies inequality $(\ref{eq4.10})$. By an identical argument as in  Step 2, we  get the desired conclusion.

		{\bf Step 4.} Suppose that the generator $g$ satisfies  $(ii)$ of assumption (UN), and  ${\rm{d}}\mathbb{P}\times {\rm{d}}t-a.e.$, $g(t,Y_{t}^{\prime},Z_{t}^{\prime})\leq g^{\prime}(t,Y_{t}^{\prime},Z_{t}^{\prime})$. Let us set for each $t\in[0,T]$,
		$$
		\hat{\xi}:=-\xi,~~~\hat{Y_{t}}:=-Y_{t},~~~\hat{Z_{t}}:=-Z_{t},~~~\hat{g}(t,y,z):=-g(t,-y,-z)
		$$
		and
		$$
		\hat{\xi^{\prime}}:=-\xi^{\prime},~~~\hat{Y^{\prime}_{t}}:=-Y^{\prime}_{t},~~~\hat{Z^{\prime}_{t}}:=-Z^{\prime}_{t},~~~\hat{g}^{\prime}(t,y^{\prime},z^{\prime}):=-g^{\prime}(t,-y^{\prime},-z^{\prime}).
		$$
		Then,  $(\hat{Y}_{t},\hat{Z}_{t})_{t\in[0,T]}$ and $(\hat{Y}^{\prime}_{t},\hat{Z}^{\prime}_{t})_{t\in[0,T]}$ are respectively a solution of BSDE $(\hat{\xi},\hat{g})$ and BSDE $(\hat{\xi}^{\prime},\hat{g}^{\prime})$. And it is easy to verify that $\mathbb{P}-a.s.$, $\hat{\xi}^{\prime}\leq \hat{\xi}$,  $\hat{g}$ satisfies $(i)$ of assumption (UN), and d$\mathbb{P}\times$ d$t-a.e.$,
		$$
			\hat{g}^{\prime}(t,\hat{Y}^{\prime}_{t},\hat{Z}^{\prime}_{t})
			=-g^{\prime}(t,Y^{\prime}_{t},Z^{\prime}_{t})
			\leq -g(t,Y^{\prime}_{t},Z^{\prime}_{t})
			= \hat{g}(t,\hat{Y}^{\prime}_{t},\hat{Z}^{\prime}_{t}).
		$$
		Thus, by virtue of Step 3, we know that $\mathbb{P}-a.s.$, for each $t\in[0,T]$,
		$$
		-Y^{\prime}_{t}=\hat{Y}^{\prime}_{t}\leq \hat{Y}_{t}=-Y_{t},
		$$
		which implies $Y_{t}\leq Y_{t}^{\prime}$. In the same way, we can prove the case that the generator $g^{\prime}$ satisfies  $(ii)$ of assumption (UN), and ${\rm{d}}\mathbb{P}\times {\rm{d}}t-a.e.$, $g(t,Y_{t},Z_{t})\leq g^{\prime}(t,Y_{t},Z_{t})$. The proof  is then complete.
	\end{proof}

	\subsection{Existence and uniqueness}\label{Title:4.2}
	According to Theorems \ref{The:3.2} and  \ref{The:4.2}, the following general existence and uniqueness theorem for unbounded solutions of BSDEs follows immediately, which generalizes Theorem 3.9 in Fan and Hu \cite{12}.
	
	\begin{theorem}\label{The:4.3}
		Assume that $\xi$ is a terminal value and $g$ is a generator satisfying assumptions (EX1), (EX2) and (UN). If  $|\xi|+\int_{0}^{T}f_{t}dt$ has sub-exponential moments of any order, i.e.,
		$$
		{\mathbb{E}}\left[{\rm exp}\left(p\left(|\xi|+\int_{0}^{T}f_{s}ds \right)^{\frac{2}{\alpha^{*}}} \right)\right]<+\infty,~~~~\forall~ p>1,
		$$
		then BSDE ($\xi,g$) admits a unique solution $(Y_{t},Z_{t})_{t\in[0,T]}$ such that $\sup\limits_{t\in[0,T]}|Y_{t}|$ has sub-exponential moments of any order and $Z_{\cdot}\in\mathcal{M}^{p}(0,T;\mathbb{R}^{d})$ for all $p>1$.
	\end{theorem}

	\section{Sufficient conditions ensuring the uniqueness holds} \label{Title:5}
In this section, we will put forward and prove several sufficient conditions ensuring that assumption (UN) holds,  some of which are explored at the first time even though for the finite time interval case. For narrative convenience, we first consider a special type of assumption (UN), i.e., the following ($\rm{UN}^{\prime}$).
	
	\begin{enumerate}
		\renewcommand{\theenumi}{($\rm{UN}^{\prime}$\alph{enumii})}
		\renewcommand{\labelenumi}{\theenumi}
		\item\label{UN2} $g$ satisfies either of the following two conditions:
		
		$(i)$~d$\mathbb{P}\times$ d$t-a.e.$, for each $(y_{i},z_{i})\in  \mathbb{R}^{1+d},~i=1,2$ and  $\theta \in(0,1)$,
		\begin{equation}\label{eq5.1}
			\tag{5.1}
			\begin{aligned}
				&{{\bf 1}}_{\left\{y_{1}-\theta y_{2}>0\right\}}(g(\omega,t,y_{1},z_{1})-\theta g(\omega,t,y_{2},z_{2}))\\
				&~\leq (1-\theta)(f_{t}(\omega)+\beta(t)|y_{2}|+\beta(t)|\delta_{\theta}y|+\gamma(t)|\delta_{\theta}z|^{\alpha});\\	
			\end{aligned}
		\end{equation}	
		
		$(ii)$~d$\mathbb{P}\times$ d$t-a.e.$, for each $(y_{i},z_{i})\in  \mathbb{R}^{1+d},~i=1,2$ and $\theta\in(0,1)$,
		\begin{equation}\label{eq5.2}
			\tag{5.2}
			\begin{aligned}
				&-{{\bf 1}}_{\left\{y_{1}-\theta y_{2}<0\right\}}(g(\omega,t,y_{1},z_{1})-\theta g(\omega,t,y_{2},z_{2}))\\
				&~\leq (1-\theta)(f_{t}(\omega)+\beta(t)|y_{2}|+\beta(t)|\delta_{\theta}y|+\gamma(t)|\delta_{\theta}z|^{\alpha}),\\
			\end{aligned}
		\end{equation}
		where  $\delta_{\theta}y$ and $\delta_{\theta}z$ defined in assumption (UN).	
	\end{enumerate}
	
	The following Propositions \ref{Pro:1} and \ref{Pro:2} present some sufficient conditions ensuring that assumption ($\rm{UN}^{\prime}$) holds, which  are  the first two main results of this section.

	\begin{proposition}\label{Pro:1}
		Assume that the generator $g$ satisfies assumption
		(EX1). Then, assumption ($UN^{\prime}$) holds for $g$ if it satisfies either of the following two conditions:
		
		$(i)~{\rm{d}}\mathbb{P}\times {\rm{d}}t-a.e.,~g(\omega,t,\cdot,\cdot)$ is convex or concave;

		$(ii)~g(t,y,z)=\gamma(t)l(y)q(z)$, where both $l: \mathbb{R}\mapsto \mathbb{R}$ and $q:\mathbb{R}^{d}\mapsto \mathbb{R}$ are  bounded Lipschitz continuous functions, the function $q(z)$ has a bounded support and $\gamma(t)\in L^{1}([0,T];\mathbb{R}_{+})\cap L^{\frac{2}{2-\alpha}}([0,T];\mathbb{R}_{+})$.
	\end{proposition}
	
	The proof of Proposition \ref{Pro:1} is similar to that of (i) and (iii) of Proposition 3.5 in Fan and Hu \cite{12}.  We omit it here.

	\begin{remark}\label{Re:5.1}
		Let $g(\omega,t,y,z):\equiv g_{1}(y)+g_{2}(y)$ and $\bar{g}(\omega,t,y,z):\equiv g_{3}(z)+g_{4}(z)$ be defined as in page 38 of Fan and Hu \cite{12}. Then, it is not difficult to verify that $\beta(t)g(\omega,t,y,z)$ and $\gamma(t)\bar{g}(\omega,t,y,z)$ satisfy  assumption ($\rm{UN}^{\prime}$). In addition, it is clear  that if $g(\omega,t,y,z)$ satisfies (i) of ($\rm{UN}^{\prime}$), then $\hat{g}(\omega,t,y,z):=-g(\omega,t,-y,-z)$ satisfies (ii) of ($\rm{UN}^{\prime}$).
	\end{remark}
	
	   Next, let $a\geq 0$, $d=1$, $u(\cdot),k_{1}(\cdot),k_{2}(\cdot)\in L^{1}([0,T];\mathbb{R}_{+})$,  $v(\cdot),c_{1}(\cdot),c_{2}(\cdot),c_{3}(\cdot)\in L^{1}([0,T];\mathbb{R}_{+})\cap L^{\frac{2}{2-\alpha}}([0,T];\mathbb{R}_{+})$. In order to explore other sufficient conditions ensuring that assumption ($\rm{UN}^{\prime}$) holds, we  introduce the following assumptions on the generator $g$.
	
	\begin{enumerate}
		\renewcommand{\theenumi}{(A\arabic{enumi})}
		\renewcommand{\labelenumi}{\theenumi}
		\item\label{A1} $g$ has a time-varying linear growth in $y$ and a time-varying sub-quadratic growth in $z$, i.e.,  ${\rm{d}}\mathbb{P}\times {\rm{d}}t-a.e.,$ for each $(y,z)\in \mathbb{R}\times \mathbb{R}$, we have
		$$
		|g(\omega,t,y,z)|\leq  f_{t}(\omega)+u(t)|y|+v(t)|z|^{\alpha}.
		$$

		\item\label{A2}  ${\rm{d}}\mathbb{P}\times {\rm{d}}t-a.e.,$ for each $z\in\mathbb{R}$, we have
		
		(i) $g(\omega,t,\cdot,z)$ satisfies a time-varying monotonicity condition on $\mathbb{R}_{-}$, i.e.,
		$$
		\forall~ (y_{1},y_{2})\in \mathbb{R}_{-} \times \mathbb{R}_{-},~~{\rm{sgn}}(y_{1}-y_{2})(g(\omega,t,y_{1},z)-g(\omega,t,y_{2},z))\leq k_{1}(t)|y_{1}-y_{2}|.
		$$

		(ii) $g(\omega,t,\cdot,z)$ satisfies a time-varying Lipschitz continuity condition on $\mathbb{R}_{+}$, i.e.,
		$$
		\forall~ (y_{1},y_{2})\in \mathbb{R}_{+} \times \mathbb{R}_{+},~~|g(\omega,t,y_{1},z)-g(\omega,t,y_{2},z)|\leq k_{2}(t)|y_{1}-y_{2}|.
		$$
		
		\item\label{A3}  ${\rm{d}}\mathbb{P}\times {\rm{d}}t-a.e.,$ for each $y\in\mathbb{R}$, we have
		
		(i) $g(\omega,t,y,\cdot)$ is a time-varying Lipschitz continuous function on $[-a,a]$, i.e.,
		$$
		\forall~ (z_{1},z_{2})\in  [-a,a] \times [-a,a],~~|g(\omega,t,y,z_{1})-g(\omega,t,y,z_{2})|\leq c_{1}(t)|z_{1}-z_{2}|.
		$$
		
		(ii) $g(\omega,t,y,\cdot)$ is  convex on $\left(-\infty,-a\right]$ and $\left[a,+\infty\right)$, respectively.
		
		\item\label{A4}  ${\rm{d}}\mathbb{P}\times {\rm{d}}t-a.e.,$ for each $y\in\mathbb{R}$, we have
		$$
		\forall~z\in  \left[a,+\infty\right),~~~g(\omega,t,y,z)-g(\omega,t,y,a)\geq -c_{2}(t)(z-a)
		$$
		and
		$$
		\forall~z\in  \left(-\infty,-a\right],~~~g(\omega,t,y,z)-g(\omega,t,y,-a)\geq c_{3}(t)(z+a).
		$$
	\end{enumerate}

	\begin{remark}\label{Re:5.2}
		$(i)$ If $g(\omega,t,\cdot,z)$ satisfies the time-varying Lipschitz continuity condition on $\mathbb{R}$, i.e.,  ${\rm{d}}\mathbb{P}\times {\rm{d}}t-a.e.$, for each $z\in\mathbb{R}$, we have
		$$
		\forall~(y_{1},y_{2})\in\mathbb{R} \times \mathbb{R},~~~|g(\omega,t,y_{1},z)-g(\omega,t,y_{2},z)|\leq \beta(t)|y_{1}-y_{2}|,
		$$
		then $g(\omega,t,\cdot,z)$  satisfies assumption (A2). Consequently, assumption (A2) generalizes the time-varying Lipschitz continuity condition of the generator $g$ with respect to $y$.
		
		$(ii)$ If $g(\omega,t,y,\cdot)$ is  convex on $\mathbb{R}$, then $g(\omega,t,y,\cdot)$  satisfies  assumption (A3) with $a=0$. Consequently, assumption (A3) generalizes the convexity condition of the generator $g$ with respect to $z$.
		
		$(iii)$ If $g(\omega,t,y,\cdot)$ satisfies (ii) of (A3), and has a right derivative at  point $a$  with a lower bound $-c_{2}(t)$ and a left derivative at  point $-a$ with a supper bound $c_{3}(t)$, then $g(\omega,t,y,\cdot)$ satisfies assumption (A4).
	\end{remark}

	\begin{proposition}\label{Pro:2}
		Let $d=1$. If the generator $g$  satisfies assumptions (H1),  (A1),  (A2), (A3) and (A4), then it also satisfies (i) of assumption ($UN^{\prime}$).
	\end{proposition}
	
	\begin{proof}
		Given arbitrarily $(y_{i}, z_{i})\in \mathbb{R}\times \mathbb{R}$, $i=1,2$ and $\theta\in(0,1)$, and let $\delta_{\theta}y$ and $\delta_{\theta}z$ be defined in assumption ($\rm{UN}^{\prime}$). It follows from  assumptions (H1) and (A2) together with Lemma \ref{Le:4.4} that
		\begin{equation}\label{eq5.3}
			\tag{5.3}
			\begin{aligned}
				&{\bf{1}}_{\left\{y_{1}>\theta y_{2}\right\}}[g(\omega, t, y_{1},z_{2})-\theta g(\omega,t,y_{2},z_{2})]\\
				&~\leq (1-\theta)\left[(k_{1}(t)+ k_{2}(t))|\delta_{\theta}y|+(k_{1}(t)+k_{2}(t))|y_{2}|+g(\omega,t,y_{2},z_{2})\right].
			\end{aligned}
		\end{equation}
	    Furthermore, define
	    $$
	    \phi(\omega,t,y_{1},|z|):= f_{t}(\omega)+u(t)|y_{1}|+v(t)|z|^{\alpha},
	    $$
	    where   $\phi(\omega,t,y_{1},\cdot):\mathbb{R}_{+}\mapsto\mathbb{R}_{+}$ is a nondecreasing continuous function.
	     In view of  assumptions (H1), (A1) (A3) and (A4) and Lemma \ref{Le:4.5} together with (iii) of Remark \ref{Re:5.2}, we get that
		$$
		\begin{aligned}
			&g(\omega, t, y_{1},z_{1})-\theta g(\omega,t,y_{1},z_{2})\\&~\leq (1-\theta)\left[\phi(\omega,t,y_{1},|\delta_{\theta}z|+2a)+2 c(t)|\delta_{\theta}z|+11c(t)a+22\phi(\omega,t,y_{1},a)\right]\\
			&~= (1-\theta)\left[23f_{t}(\omega)+23u(t)|y_{1}|+v(t)||\delta_{\theta}z|+2a|^{\alpha}+22v(t)a^{\alpha}+2c(t)|\delta_{\theta}z|+11c(t)a\right]\\
			&~\leq (1-\theta)\left[23f_{t}(\omega)+23u(t)|y_{1}|+2^{\alpha}v(t)(|\delta_{\theta}z|^{\alpha}+(2a)^{\alpha})+22v(t)a^{\alpha}+2c(t)|\delta_{\theta}z|+11c(t)a\right]\\
			&~\leq (1-\theta)\left[23f_{t}(\omega)+23u(t)|y_{1}-\theta y_{2}|+23u(t)\theta|y_{2}|+4v(t)|\delta_{\theta}z|^{\alpha}+38v(t)a^{\alpha}+2c(t)|\delta_{\theta}z|+11c(t)a\right]\\
			&~\leq (1-\theta)\left[\bar{f}_{t}(\omega)+23u(t)|y_{2}|+23u(t)|\delta_{\theta}y|+(4v(t)+2 c(t))|\delta_{\theta}z|^{\alpha}\right],
		\end{aligned}
		$$
		where $c(t):=c_{1}(t)\vee c_{2}(t)\vee c_{3}(t)$, and the process $\bar{f}_{t}(\omega):=23f_{t}(\omega)+38v(t)a^{\alpha}+2c(t)+11c(t)a$ has sub-exponential moments of any order.
		Combining inequality (\ref{eq5.3}) and the last inequality, we have
		$$
		\begin{aligned}
			&{\bf{1}}_{\left\{y_{1}>\theta y_{2}\right\}}[g(\omega, t, y_{1},z_{1})-\theta g(\omega,t,y_{2},z_{2})]\\
			&~= {\bf{1}}_{\left\{y_{1}>\theta y_{2}\right\}}[g(\omega, t, y_{1},z_{1})-\theta g(\omega, t, y_{1},z_{2})]+{\bf{1}}_{\left\{y_{1}>\theta y_{2}\right\}}\theta[g(\omega, t, y_{1},z_{2})-\theta g(\omega, t, y_{2},z_{2})]\\
			&~~~~~-{\bf{1}}_{\left\{y_{1}>\theta y_{2}\right\}}\theta(1-\theta)g(\omega, t, y_{2},z_{2})\\
			&~\leq  {\bf{1}}_{\left\{y_{1}>\theta y_{2}\right\}}[g(\omega, t, y_{1},z_{1})-\theta g(\omega, t, y_{1},z_{2})]+(1-\theta)\theta \left[(k_{1}(t)+ k_{2}(t))|\delta_{\theta}y|+(k_{1}(t)+k_{2}(t))|y_{2}|\right]\\
			&~\leq (1-\theta) [\bar{f}_{t}(\omega)+(23u(t)+k_{1}(t)+k_{2}(t))|y_{2}|+(23u(t)+k_{1}(t)+ k_{2}(t))|\delta_{\theta}y|+(4v(t)+2c(t))|\delta_{\theta}z|^{\alpha}].
		\end{aligned}
		$$
		Then, the desired inequality $(\ref{eq5.1})$ holds with $f_{t}(\omega):=\bar{f}_{t}(\omega)$, $\beta(t):=23u(t)+k_{1}(t)+k_{2}(t)$ and $\gamma(t):=v(t)+  2c(t)$. The proof is then complete.
	\end{proof}

	Next, let $d=1$,  $\bar{u}(\cdot)\in L^{1}([0,T];\mathbb{R}_{+})$, and $ \bar{c}(\cdot), \bar{v}(\cdot)\in L^{\frac{2}{2-\alpha}}([0,T];\mathbb{R}_{+})$ satisfying $0<\int_{0}^{T}\bar{c}(t)dt<+\infty$ and  $0<\int_{0}^{T}\bar{v}(t)dt<+\infty$. In order to explore sufficient conditions ensuring that assumption (UN) holds, we  introduce the following assumptions on the generator $g$.
	
	\begin{enumerate}
		\renewcommand{\theenumi}{(A5\alph{enumii})}
		\renewcommand{\labelenumi}{\theenumi}
		\item\label{A5}  $g$ has a time-varying linear growth in $y$ and a time-varying logarithmic growth in $z$, i.e.,  ${\rm{d}}\mathbb{P}\times {\rm{d}}t-a.e.,$ for each $(y,z)\in \mathbb{R}\times \mathbb{R}$, we have
		$$
		|g(\omega,t,y,z)|\leq  f_{t}(\omega)+\bar{u}(t)|y|+\bar{v}(t)\left[{\rm{ln}}\left(e+|z|\right)\right]^{\frac{\alpha^{*}}{2}}.
		$$
	\end{enumerate}

	\begin{enumerate}
		\renewcommand{\theenumi}{(A6\alph{enumii})}
		\renewcommand{\labelenumi}{\theenumi}
		\item\label{A6} There exists a constant $a\geq 0$ such that ${\rm{d}}\mathbb{P}\times {\rm{d}}t-a.e.,$ for each $y\in\mathbb{R}$, we have
		
		(i)  $g(\omega,t,y,\cdot)$ is a time-varying Lipschitz continuous function on $\mathbb{R}$, i.e.,
		$$
		\forall~ (z_{1},z_{2})\in  \mathbb{R} \times \mathbb{R},~~|g(\omega,t,y,z_{1})-g(\omega,t,y,z_{2})|\leq \bar{c}(t)|z_{1}-z_{2}|.
		$$
		
		(ii) $g(\omega,t,y,\cdot)$ decreases  on $\left(-\infty,-a\right]$ and increases  on $\left[a,+\infty\right)$.
	\end{enumerate}

	The following Proposition \ref{Pro:3} is the third main result of this section.
	
	\begin{proposition}\label{Pro:3}
		Let $d=1$. If the generator $g$ satisfies assumptions (H1), (A2),  (A5) and (A6), then it also satisfies (i) of assumption (UN).
	\end{proposition}
	\begin{proof}
		Given arbitrarily $(y_{i}, z_{i})\in \mathbb{R}\times \mathbb{R}$, $i=1,2$ and $\theta\in(0,1)$, and let $\delta_{\theta}y$ and $\delta_{\theta}z$ be defined in assumption (UN). Define
		$$
		\phi(\omega,t,y_{1},|z|):=f_{t}(\omega)+\bar{u}(t)|y_{1}|+\bar{v}(t)\left[{\rm{ln}}\left(e+|z|\right)\right]^{\frac{\alpha^{*}}{2}},
		$$
		where  $\phi(\omega,t,y_{1},\cdot):\mathbb{R}_{+}\mapsto\mathbb{R}_{+}$ is a nondecreasing continuous function.
		In view of assumptions (H1), (A5) and  (A6),   and Lemma \ref{Le:4.6}, we get
		$$
		\begin{aligned}
			&g(\omega, t, y_{1},z_{1})-\theta g(\omega,t,y_{1},z_{2})\\&~\leq (1-\theta)\left[4{\bar{c}}(t)|\delta_{\theta}z|+4\bar{c}(t)a+11\phi(\omega,t,y_{1},a)+\phi(\omega,t,y_{1},|z_{2}|)\right]\\
				&~= (1-\theta)\left[12f_{t}(\omega)+12\bar{u}(t)|y_{1}|+4\bar{c}(t)|\delta_{\theta}z|+4\bar{c}(t)a+\bar{v}(t)\left[{\rm{ln}}\left(e+|z_{2}|\right)\right]^{\frac{\alpha^{*}}{2}}+11\bar{v}(t)\left[{\rm{ln}}\left(e+a\right)\right]^{\frac{\alpha^{*}}{2}}\right]\\
			&~\leq (1-\theta)\left[12f_{t}(\omega)+12\bar{u}(t)|y_{1}-\theta y_{2}|+12\bar{u}(t)\theta |y_{2}|+4\bar{c}(t)|\delta_{\theta}z|+4\bar{c}(t)a+\bar{v}(t)\left[{\rm{ln}}\left(e+|z_{2}|\right)\right]^{\frac{\alpha^{*}}{2}}\right.\\
			&~~~~~~~~~~~~~~~\left.+11\bar{v}(t)\left[{\rm{ln}}\left(e+a\right)\right]^{\frac{\alpha^{*}}{2}}\right]\\
			&~\leq (1-\theta)\left[\hat{f}_{t}(\omega)+12\bar{u}(t)|y_{2}|+12\bar{u}(t)|\delta_{\theta}y|+\bar{v}(t)\left[{\rm{ln}}\left(e+|z_{2}|\right)\right]^{\frac{\alpha^{*}}{2}}+4\bar{c}(t)|\delta_{\theta}z|^{\alpha} \right],
		\end{aligned}
		$$
		where the process $\hat{f}_{t}(\omega):=12f_{t}(\omega)+11\bar{v}(t)\left[{\rm{ln}}\left(e+a\right)\right]^{\frac{\alpha^{*}}{2}}+4\bar{c}(t)a+4\bar{c}(t)$ has sub-exponential moments of any order. Furthermore, in view of assumptions (H1) and (A2), it follows from Lemma \ref{Le:4.4} that the inequality (\ref{eq5.3}) holds.	Combining inequality (\ref{eq5.3}) and the last inequality, we have
		$$
		\begin{aligned}
			&{\bf{1}}_{\left\{y_{1}>\theta y_{2}\right\}}[g(\omega, t, y_{1},z_{1})-\theta g(\omega,t,y_{2},z_{2})]\\
			&= {\bf{1}}_{\left\{y_{1}>\theta y_{2}\right\}}[g(\omega, t, y_{1},z_{1})-\theta g(\omega, t, y_{1},z_{2})]+{\bf{1}}_{\left\{y_{1}>\theta y_{2}\right\}}\theta[g(\omega, t, y_{1},z_{2})-\theta g(\omega, t, y_{2},z_{2})]\\
			&~~~~~-{\bf{1}}_{\left\{y_{1}>\theta y_{2}\right\}}\theta(1-\theta)g(\omega, t, y_{2},z_{2})\\
			&\leq {\bf{1}}_{\left\{y_{1}>\theta y_{2}\right\}} [g(\omega, t, y_{1},z_{1})-\theta g(\omega, t, y_{1},z_{2})]+(1-\theta)\theta\left[(k_{1}(t)+ k_{2}(t))|\delta_{\theta}y|+(k_{1}(t)+k_{2}(t))|y_{2}|\right]\\
			&\leq (1-\theta) \Big[\hat{f}_{t}(\omega)+(k_{1}(t)+k_{2}(t)+12\bar{u}(t))|y_{2}|+\bar{v}(t)\left[{\rm{ln}}\left(e+|z_{2}|\right)\right]^{\frac{\alpha^{*}}{2}}\Big.\\
			&~~~~~~~~~~~~~~\Big.+(k_{1}(t)+ k_{2}(t)+12\bar{u}(t))|\delta_{\theta}y|+ 4\bar{c}(t)|\delta_{\theta}z|^{\alpha}\Big].
		\end{aligned}
		$$
		Thus, the desired inequality $(\ref{eq4.1})$ holds with $f_{t}(\omega):=\hat{f}_{t}(\omega)$, $\beta(t):=k_{1}(t)+k_{2}(t)+12\bar{u}(t)$ and $\gamma(t):=\bar{v}(t)+4\bar{c}(t)$. The proof is then complete.
	\end{proof}

	\begin{remark}\label{Re:5.3}
		Let $d\geq 2$. It is not hard to verify that for each $i=1,2,\dots,d$, if the function $g^{i}(\cdot):\mathbb{R}\rightarrow \mathbb{R}$ satisfies assumption (UN) (resp.($\rm{UN}^{\prime}$)) with $\mathbb{R}^{2}$ instead of $\mathbb{R}^{1+d}$, then $g(\omega,t,y,z):=\sum\limits_{i=1}^{d}g^{i}(z^{i})$  satisfies assumption (UN) (resp.($\rm{UN}^{\prime}$)), where $z=(z^{1},z^{2},\dots,z^{d})^{ \mathrm{T}}\in \mathbb{R}^{d}$.
	\end{remark}
	
	The following examples show that Theorem \ref{The:4.3}  is not covered by any existing results.
	
	\begin{example}\label{Exa:2} 	Let $\beta(\cdot)\in L^{1}([0,T];\mathbb{R}_{+})$, $\gamma(\cdot)\in L^{1}([0,T];\mathbb{R}_{+})\cap L^{\frac{2}{2-\alpha}}([0,T];\mathbb{R}_{+})$. For $(\omega,t,y,z)\in \Omega \times [0,T] \times \mathbb{R}\times \mathbb{R}^{d}$ with $z=(z^{1},z^{2},\dots,z^{d})^{ \mathrm{T}}$,  define
		$$
		g(\omega,t,y,z):=|B_{t}(\omega)|+\beta(t)\left(\sqrt[3]{y}{\bf 1}_{y\leq 0}+{\rm{sin}}y{\bf 1}_{y> 0}\right)+\gamma(t)\sum\limits_{i=1}^{d}q(z^{i})+\gamma(t)|z|^{\alpha},$$
		where
		$$
		\begin{aligned}
			\begin{split}
				q(x):=\left\{
				\begin{array}{ll}
					-x, & x\leq -2;\\
					-\frac{3}{2}x-1, & -2<x<2;\\
					-2x,  & x\geq 2.\\
				\end{array}
				\right.
			\end{split}
		\end{aligned}
		$$
		In view of Proposition $\ref{Pro:2}$, (iii) of Remark \ref{Re:4.1} and Remark \ref{Re:5.3}, it is not difficult to verify that the  generator $g$ satisfies assumptions  (EX1), (EX2) and ($\rm{UN}^{\prime}$),  but it satisfies neither the time-varying Lipschitz continuity condition nor the convexity/concavity condition with respect to $(y,z)$. In particular, we note that $g$ is  not   time-varying locally Lipschitz continuous in  $y$.
	\end{example}

	\begin{example}\label{Exa:4}
		Let $\beta(\cdot)\in L^{1}([0,T];\mathbb{R}_{+})$, $\gamma(\cdot)\in L^{\frac{2}{2-\alpha}}([0,T];\mathbb{R}_{+})$ satisfying inequality (\ref{eq4.3}). For $(\omega,t,y,z)\in \Omega \times [0,T] \times \mathbb{R}\times \mathbb{R}^{d}$ with $z=(z^{1},z^{2},\dots,z^{d})^{ \mathrm{T}}$,  define
		$$
		g(\omega,t,y,z):=|B_{t}(\omega)|+\beta(t)\sqrt{|y|}{\bf 1}_{y\leq 0}+\gamma(t)\sum\limits_{i=1}^{d}\left[{\rm{ln}}(e+|z^{i}|)\right]^{\frac{\alpha^{*}}{2}}+2\gamma(t)|z|^{\alpha}.
		$$
		According to Proposition $\ref{Pro:3}$, (iii) of Remark \ref{Re:4.1} and Remark \ref{Re:5.3},  it is easy to verify that the generator $g$ satisfies assumptions  (EX1), (EX2) and (UN), but it satisfies neither the time-varying Lipschitz  continuity condition nor the convexity/concavity condition with respect to $(y,z)$. In particular, this $g$ is also not time-varying locally Lipschitz continuous in $y$.
	\end{example}

	\section*{Appendix}
     \renewcommand{\thelemma}{A.\arabic{lemma}}

	\begin{lemma}\label{Le:4.4}
		
		Let  $f(x): \mathbb{R}\mapsto \mathbb{R}$ be a  continuous function.
	Assume that there exist two constants $k_{1},~k_{2}>0$ such that
	
	(i) $f(x)$ satisfies a monotonicity condition on $\mathbb{R}_{-}$, i.e.,
	$$
	\forall~ (x_{1},x_{2})\in \mathbb{R}_{-} \times \mathbb{R}_{-},~~{\rm{sgn}}(x_{1}-x_{2})(f(x_{1})-f(x_{2}))\leq k_{1}|x_{1}-x_{2}|.
	$$
	
	(ii) $f(x)$ is Lipschitz continuous on $\mathbb{R}_{+}$, i.e.,
	$$
	\forall~ (x_{1},x_{2})\in \mathbb{R}_{+} \times \mathbb{R}_{+},~~|f(x_{1})-f(x_{2})|\leq k_{2}|x_{1}-x_{2}|.
	$$
	Then, for each $(x_{1},x_{2})\in \mathbb{R}\times \mathbb{R}$ and $\theta\in(0,1)$, we have
	\begin{equation}\label{eq:A.1}
		\tag{A.1}
		{\bf{1}}_{\left\{x_{1}>\theta x_{2}\right\}}\frac{f(x_{1})-\theta f(x_{2})}{1-\theta}\leq (k_{1}+ k_{2})\left|\frac{x_{1}-\theta x_{2}}{1-\theta}\right|+(k_{1}+k_{2})|x_{2}|+f(x_{2}).
	\end{equation}
	\end{lemma}

	\begin{proof}
		Given arbitrarily $(x_{1},x_{2})\in \mathbb{R}\times \mathbb{R}$ and $\theta\in(0,1)$. Note that the inequality (\ref{eq:A.1}) is clearly true when $x_{1}\leq \theta x_{2}$. We only need to consider the following three cases: $x_{1}>\theta x_{2}\geq 0$, $0\geq x_{1}>\theta x_{2}>x_{2}$ and $x_{1}>0>\theta x_{2}$.
		
		{\bf (a) The case of $x_{1}>\theta x_{2}\geq 0$:} In view of (ii), we have
		$$
		\begin{aligned}
			f(x_{1})-\theta f(x_{2})&\leq|f(x_{1})-f( x_{2})|+(1-\theta)f(x_{2})\leq k_{2}|x_{1}- x_{2}|+(1-\theta)f(x_{2})\\
			&\leq (1-\theta)\left[k_{2}\left|\frac{x_{1}-\theta x_{2}}{1-\theta}\right|+k_{2}|x_{2}|+f(x_{2})\right].
		\end{aligned}
		$$
		
		{\bf (b) The case of $0\geq x_{1}>\theta x_{2}>x_{2}$:} Using (i), we get
		$$
		\begin{aligned}
			f(x_{1})-\theta f(x_{2})&=f(x_{1})- f(x_{2})+(1-\theta)f(x_{2})\leq k_{1}|x_{1}-x_{2}| +(1-\theta)f(x_{2})\\
			&\leq (1-\theta)\left[k_{1}\left|\frac{x_{1}-\theta x_{2}}{1-\theta}\right|+k_{1}|x_{2}|+f(x_{2})\right].
		\end{aligned}
		$$
		
		{\bf (c) The case of $x_{1}>0>\theta x_{2}$:} It follows from $(b)$ that
		$$
		f(0)-\theta f(x_{2})\leq (1-\theta)\left[k_{1}\left|\frac{-\theta x_{2}}{1-\theta}\right|+k_{1}|x_{2}|+f(x_{2})\right].
		$$
		Then, combining (ii) and the last inequality, we have
		$$
		\begin{aligned}
			f(x_{1})-\theta f( x_{2})&= f(x_{1})-f(0)+f(0)-\theta f(x_{2})\\
             &\leq k_{2}|x_{1}| +(1-\theta)\left[k_{1}\left|\frac{-\theta x_{2}}{1-\theta}\right|+k_{1}|x_{2}|+f(x_{2})\right]\\
			&=(1-\theta)\left[(k_{1}+ k_{2})\left|\frac{x_{1}-\theta x_{2}}{1-\theta}\right|+k_{1}|x_{2}|+f(x_{2})\right].
		\end{aligned}
		$$
		All in all, the inequality (\ref{eq:A.1}) holds.
	\end{proof}

	\begin{lemma}\label{Le:4.5}
		Let  $f(x): \mathbb{R}\mapsto \mathbb{R}$  be a continuous function such that $|f(x)|\leq \phi(|x|)$ for each $x\in \mathbb{R}$, where $\phi(x):\mathbb{R}_{+}\mapsto\mathbb{R}_{+}$ is a nondecreasing continuous function. Assume that there exist two constants $a\geq 0$ and $k>0$ such that
		
		(i) $f(x)$ is  Lipschitz continuous on $[-a,a]$, i.e.,
		$$
		\forall~ (x_{1},x_{2})\in [-a,a] \times [-a,a],~~|f(x_{1})-f(x_{2})|\leq k|x_{1}-x_{2}|;
		$$
		
		(ii) $f(x)$ is  convex on $\left(-\infty,-a\right]$ and $\left[a,+\infty\right)$, respectively;

        (iii) Both $f_{-}^{\prime}(-a)$ and $f_{+}^{\prime}(a)$ are finite value.\\
		Then, for each $(x_{1},x_{2})\in \mathbb{R}\times \mathbb{R}$ and $\theta\in(0,1)$, we have
		\begin{equation}\label{eq:A.2}
			\tag{A.2}
			\frac{f(x_{1})-\theta f(x_{2})}{1-\theta}\leq \phi\left(\left|\frac{x_{1}-\theta x_{2}}{1-\theta}\right|+2a\right)+2k_{0}\left|\frac{x_{1}-\theta x_{2}}{1-\theta}\right|+11k_{0}a+22\phi(a),
		\end{equation}
	\end{lemma}
	\noindent where $k_{0}:= |f_{-}^{\prime}(-a)|\vee |f_{+}^{\prime}(a)|\vee k.$

	\begin{proof} Without loss of generality, we can assume that $a>0$. The proof is divided into the following four steps.
		
		{\bf  Step 1.}  We construct a  continuous function $g(x)$ such that
		
		$\bullet~ g(x)=f(x)$, $\forall~ x\in \left(-\infty,-a\right]\cup \left[a,+\infty\right)$;
		
		$\bullet~ g(x)$ is respectively convex on $\left(-\infty,x_{0}\right]$ and $\left[x_{0},+\infty\right)$ for some $x_{0}\in(-a,a)$.\\
		In fact, we can define $g(x)$ as follows:
		$$
		g(x)=\begin{cases}
			f(x),& x \in (-\infty,-a]; \\
			k_{0}(x+a)+f(-a),& x\in (-a,x_{0}];\\
			-k_{0}(x-a)+f(a), &x\in [x_{0},a);\\
			f(x), & x\in [a,+\infty),
		\end{cases}
		$$ \noindent where, in view of (i),
		$$
		x_{0}:=\frac{f(-a)-f(a)}{2k_{0}}\in(-a,a).
		$$
		 By the definition of $g(x)$, we get
		\begin{equation}\label{eq:A.3}
			\tag{A.3}
			\forall~ x\in \mathbb{R},~~~|g(x)|\leq|f(x)|+k_{0}a+2\sup\limits_{[-a,a]}|f(x)|\leq \phi(|x|)+k_{0}a+2\phi(a).
		\end{equation}

		\begin{figure}[H]
			\centering
			\includegraphics[width=12.5cm,height=5.4cm]{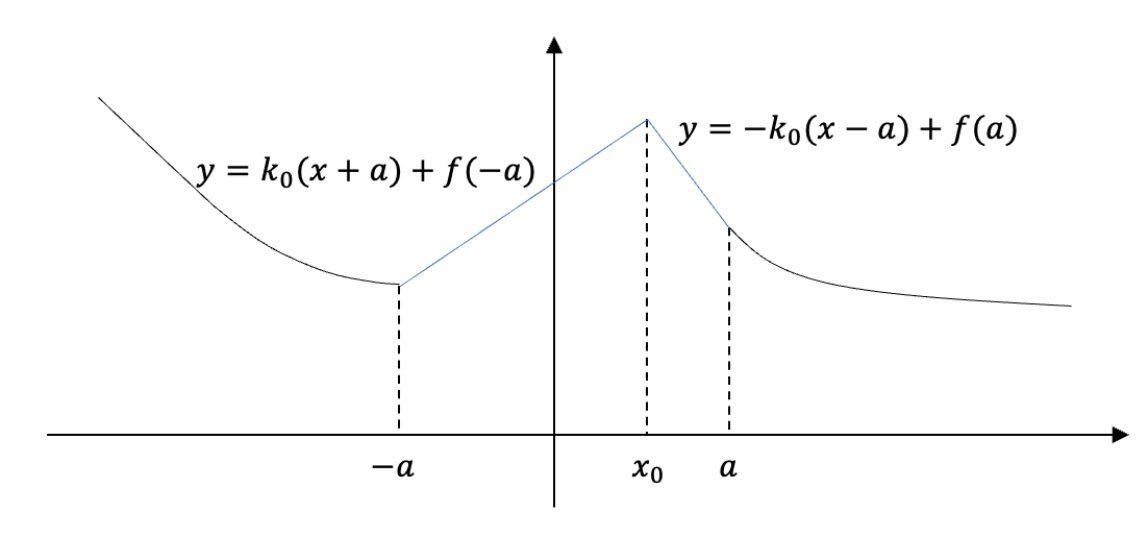}
			\caption{a graph of $g(x)$}
			\label{1}
		\end{figure}

		{\bf Step 2.} Define $\bar{g}(x):= g(x+x_{0})-g(x_{0})$. Given arbitrarily $(x_{1},x_{2})\in \mathbb{R}\times \mathbb{R}$ and $\theta\in(0,1)$, we prove that $\bar{g}(x)$ satisfies the following inequality
		\begin{equation}\label{eq:A.4}
			\tag{A.4}
			\frac{\bar{g}(x_{1})-\theta \bar{g}( x_{2})}{1-\theta}\leq \phi\left(\left|\frac{x_{1}-\theta x_{2}}{1-\theta}\right|+a\right)+k_{0}\left|\frac{x_{1}-\theta x_{2}}{1-\theta}\right|+2k_{0}a+5\phi(a).\\
		\end{equation}
		Firstly, $\bar{g}(0)=0$ and  we have
		
		$\bullet ~\bar{g}(x)$ is Lipschitz continuous on $[-a-x_{0},a-x_{0}]$, i.e.,
		$$
		\forall~ (x_{1},x_{2})\in [-a-x_{0},a-x_{0}] \times [-a-x_{0},a-x_{0}],~~|\bar{g}(x_{1})-\bar{g}|(x_{2})|\leq k_{0}|x_{1}-x_{2}|;
		$$
		
		$\bullet~ \bar{g}(x)$  is convex on $\left(-\infty,0\right]$ and $\left[0,+\infty\right)$, respectively.\\
		On the other hand, from the definition of $\bar{g}$ and inequality (\ref{eq:A.3}), we know that
		\begin{equation}\label{eq:A.5}
			\tag{A.5}
			\forall~ x\in \mathbb{R},~~~|\bar{g}(x)|\leq |g(x+x_{0})|+|g(x_{0})|\leq \phi (|x|+a)+2k_{0}a+5\phi(a).
		\end{equation}

		\begin{figure}[H]
			\centering
			\includegraphics[width=12cm,height=6cm]{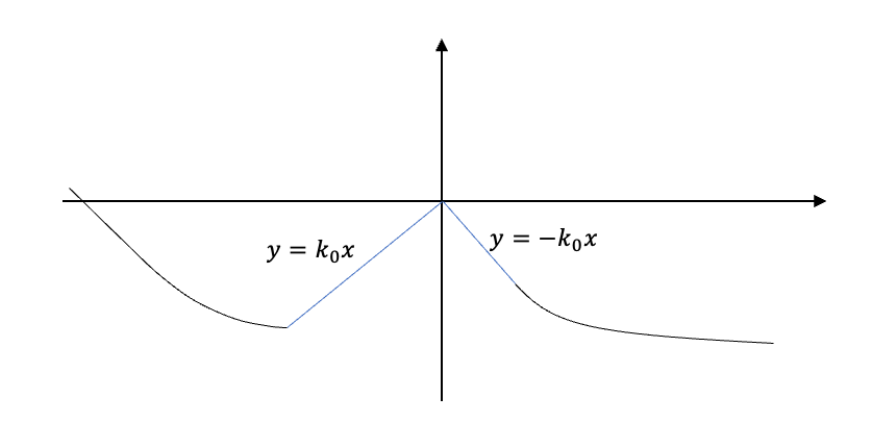}
			\caption{a graph of  $\bar{g}(x)$}
			\label{2}
		\end{figure}
		Next, we need to consider the following four cases:  $x_{1},x_{2}\geq 0$, $x_{1},x_{2}\leq 0$, $x_{2}\leq 0\leq x_{1}$ and $x_{1}\leq 0\leq x_{2}$.
		
		{\bf (a) The case of $x_{1},x_{2}\geq 0$:} Define
		$$
		\bar{g}_{1}(x):=\begin{cases}
			\bar{g}(x),& x \in [0,+\infty); \\
			-k_{0}x,& x\in (-\infty,0).\\
		\end{cases}
		$$
		\begin{figure}[H]
			\centering
			\includegraphics[width=11cm,height=6cm]{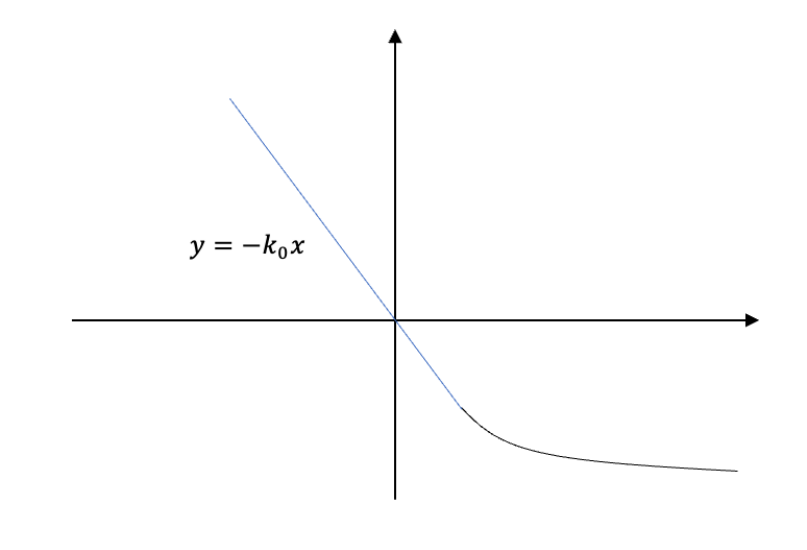}
			\caption{a graph of  $\bar{g}_{1}(x)$}
			\label{3}
		\end{figure}
		\noindent It is obvious that $\bar{g}_{1}(x)$ is convex on $\mathbb{R}$ and satisfies
		$$
		\forall~ x\in \mathbb{R},~~~~ \bar{g}_{1}(x)\leq |\bar{g}(x)|+k_{0}|x|.
		$$
		Then, we have
		$$
		\begin{aligned}
			\frac{\bar{g}(x_{1})-\theta \bar{g}( x_{2})}{1-\theta}&=\frac{\bar{g}_{1}(x_{1})-\theta \bar{g}_{1}( x_{2})}{1-\theta}=\frac{\bar{g}_{1}(\theta x_{2}+(1-\theta)\frac{x_{1}-\theta x_{2}}{1-\theta})- \theta \bar{g}_{1}(x_{2})}{1-\theta}\\
			&\leq \bar{g}_{1}\left(\frac{x_{1}-\theta x_{2}}{1-\theta}\right)\leq \left|\bar{g}\left(\frac{x_{1}-\theta x_{2}}{1-\theta}\right)\right|+k_{0}\left|\frac{x_{1}-\theta x_{2}}{1-\theta}\right|.
		\end{aligned}
		$$
		
		{\bf (b) The case of $x_{1},x_{2}\leq 0$:} Define
		$$
		\bar{g}_{2}(x):=\begin{cases}
			\bar{g}(x),& x \in (-\infty,0]; \\
			k_{0}x,& x\in (0,+\infty).\\
		\end{cases}
		$$
		\begin{figure}[H]
			\centering
			\includegraphics[width=11cm,height=6cm]{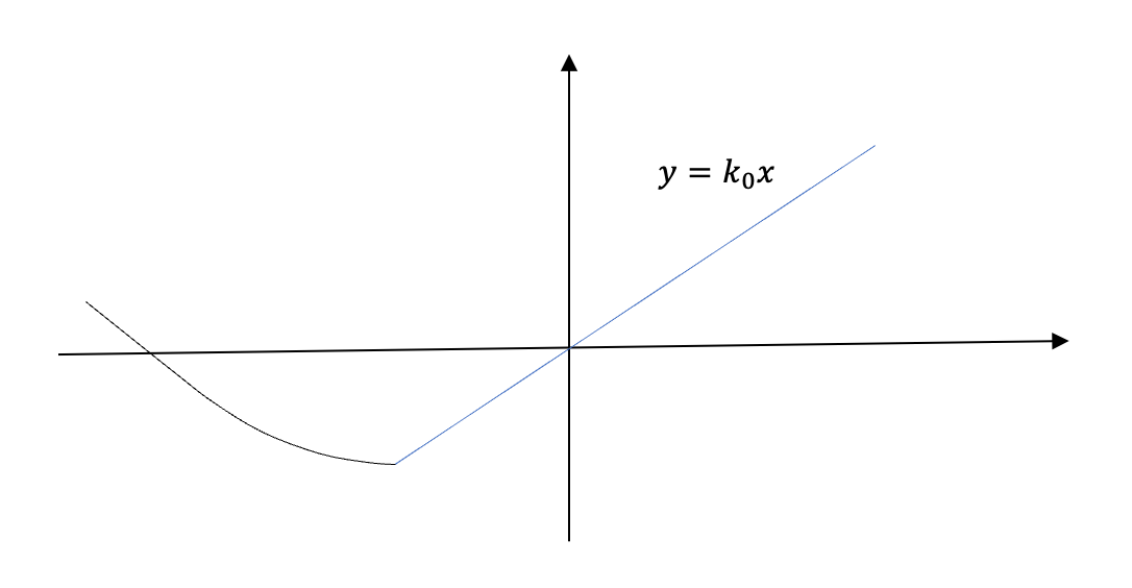}
			\caption{a graph of  $\bar{g}_{2}(x)$}
			\label{4}
		\end{figure}
		\noindent It is clear that $\bar{g}_{2}(x)$ is convex on $\mathbb{R}$ and satisfies
		$$
		\forall~ x\in \mathbb{R},~~~~ \bar{g}_{2}(x)\leq |\bar{g}(x)|+k_{0}|x|.
		$$
		Then, we have
		$$
		\begin{aligned}
			\frac{\bar{g}(x_{1})-\theta \bar{g}( x_{2})}{1-\theta}&=\frac{\bar{g}_{2}(x_{1})-\theta \bar{g}_{2}( x_{2})}{1-\theta}=\frac{\bar{g}_{2}(\theta x_{2}+(1-\theta)\frac{x_{1}-\theta x_{2}}{1-\theta})- \theta \bar{g}_{2}(x_{2})}{1-\theta}\\
			&\leq \bar{g}_{2}\left(\frac{x_{1}-\theta x_{2}}{1-\theta}\right)\leq \left|\bar{g}\left(\frac{x_{1}-\theta x_{2}}{1-\theta}\right)\right|+k_{0}\left|\frac{x_{1}-\theta x_{2}}{1-\theta}\right|.
		\end{aligned}
		$$
		
		{\bf (c) The case of $x_{2}\leq 0\leq x_{1}$:} It follows from $(a)$ that
		\begin{equation}\label{eq:A.6}
			\tag{A.6}
			\frac{\bar{g}(x_{1})-\theta \bar{g}(0)}{1-\theta}\leq \bar{g}_{1}\left(\frac{x_{1}}{1-\theta}\right)=\bar{g}\left(\frac{x_{1}}{1-\theta}\right).
		\end{equation}
		In view of $(b)$, we get
		\begin{equation}\label{eq:A.7}
			\tag{A.7}
			\frac{\bar{g}(0)-\theta \bar{g}(x_{2})}{1-\theta}\leq \bar{g}_{2}\left(\frac{-\theta x_{2}}{1-\theta}\right)=k_{0}\left|\frac{-\theta x_{2}}{1-\theta}\right|.
		\end{equation}
		Then, combining inequalities (\ref{eq:A.6}) and (\ref{eq:A.7}) together with $\bar{g}(0)=0$, we have	
		$$
		\begin{aligned}
			\frac{\bar{g}(x_{1})-\theta \bar{g}( x_{2})}{1-\theta}&=\frac{\bar{g}(x_{1})-\theta \bar{g}(0)+\bar{g}(0)-\theta\bar{g}(x_{2})}{1-\theta}\leq\bar{g}\left(\frac{x_{1}}{1-\theta}\right)+k_{0}\left|\frac{x_{1}-\theta x_{2}}{1-\theta}\right|.\\
		\end{aligned}
		$$
		
		{\bf (d) The case of $x_{1}\leq 0\leq x_{2}$:} It follows from $(b)$ that
		\begin{equation}\label{eq:A.8}
			\tag{A.8}
			\frac{\bar{g}(x_{1})-\theta \bar{g}(0)}{1-\theta}\leq \bar{g}_{2}\left(\frac{x_{1}}{1-\theta}\right)=\bar{g}\left(\frac{x_{1}}{1-\theta}\right).
		\end{equation}
		By virtue of $(a)$, we have
		\begin{equation}\label{eq:A.9}
			\tag{A.9}
			\frac{\bar{g}(0)-\theta \bar{g}(x_{2})}{1-\theta}\leq \bar{g}_{1}\left(\frac{-\theta x_{2}}{1-\theta}\right)=k_{0}\left|\frac{-\theta x_{2}}{1-\theta}\right|.
		\end{equation}
		Then, combining inequalities (\ref{eq:A.8}) and (\ref{eq:A.9}) together with $\bar{g}(0)=0$, we get	
		$$
		\begin{aligned}
			\frac{\bar{g}(x_{1})-\theta \bar{g}( x_{2})}{1-\theta}&=\frac{\bar{g}(x_{1})-\theta \bar{g}(0)+\bar{g}(0)-\theta\bar{g}(x_{2})}{1-\theta}\leq\bar{g}\left(\frac{x_{1}}{1-\theta}\right)+k_{0}\left|\frac{x_{1}-\theta x_{2}}{1-\theta}\right|.\\
		\end{aligned}
		$$
		All in all, in view of inequality (\ref{eq:A.5}),  the inequality (\ref{eq:A.4}) holds.
		
		{\bf Step 3.} Given arbitrarily $(x_{1},x_{2})\in \mathbb{R}\times \mathbb{R}$ and $\theta\in(0,1)$, we prove  that $g(x)$ satisfies the following inequality
		\begin{equation}\label{eq:A.10}
			\tag{A.10}
			\frac{g(x_{1})-\theta g( x_{2})}{1-\theta}\leq \phi\left(\left|\frac{x_{1}-\theta x_{2}}{1-\theta}\right|+2a\right)+k_{0}\left|\frac{x_{1}-\theta x_{2}}{1-\theta}\right|+4k_{0}a+7\phi(a).
		\end{equation}
		In fact, it follows from the definition of $g(x)$ that
		$$
		|g(x_{0})|\leq k_{0}a+ 2\sup\limits_{[-a,a]}|f(x)|\leq k_{0}a+2\phi(a).
		$$
		Combining inequality (\ref{eq:A.4}) and the last inequality,  we have
		$$
		\begin{aligned}
			\frac{g(x_{1})-\theta g( x_{2})}{1-\theta}&=\frac{(g(x_{1}-x_{0}+x_{0})-g(x_{0}))-\theta (g(x_{2}-x_{0}+ x_{0})-g(x_{0}))}{1-\theta}
			+g(x_{0})\\
			&\leq \frac{\bar{g}(x_{1}-x_{0})-\theta \bar{g}( x_{2}-x_{0})}{1-\theta}+k_{0}a+2\phi(a)\\
			&\leq \phi\left(\left|\frac{(x_{1}-x_{0})-\theta (x_{2}-x_{0})}{1-\theta}\right|+a\right) +k_{0}\left|\frac{(x_{1}-x_{0})-\theta (x_{2}-x_{0})}{1-\theta}\right|+3k_{0}a+7\phi(a)\\
			&\leq \phi\left(\left|\frac{x_{1}-\theta x_{2}}{1-\theta}-x_{0}\right|+a\right)+k_{0}\left|\frac{x_{1}-\theta x_{2}}{1-\theta}\right| +k_{0}|x_{0}|+3k_{0}a+7\phi(a)\\
			&\leq \phi\left(\left|\frac{x_{1}-\theta x_{2}}{1-\theta}\right|+2a\right)+k_{0}\left|\frac{x_{1}-\theta x_{2}}{1-\theta}\right|+4k_{0}a+7\phi(a).\\
		\end{aligned}
		$$
		Hence, the inequality (\ref{eq:A.10}) holds.
		
		{\bf Step 4.} We prove that $f(x)$  satisfies (\ref{eq:A.2}). Set $h(x):=f(x)-g(x)$. Then $h(x)$ satisfies (i)  with   Lipschitz constant $ k_{0}$, and  $h(x)\equiv 0$ when $|x|\geq a$. Note that $h(x)$ has a bound $M:=k_{0}a+3\sup\limits_{[-a,a]}|f(x)|\leq k_{0}a+3\phi(a)$,  for each $(x_{1},x_{2})\in\mathbb{R} \times \mathbb{R}$ and $\theta\in(0,1)$, we have
		$$
		\begin{aligned}
			|h(\theta x_{2})-h(x_{2})|&=|h(\theta x_{2})-h(x_{2})|{\bf{1}}_{\left\{\theta\in\left(0,1/2\right]\right\}}+|h(\theta x_{2})-h(x_{2})|{\bf{1}}_{\left\{\theta\in\left(1/2,1\right)\right\}}\\
			&\leq (1-\theta)\frac{2M}{1-\theta}{\bf{1}}_{\left\{\theta\in\left(0,1/2\right]\right\}}+(1-\theta)k_{0}|x_{2}|{\bf{1}}_{\left\{|x_{2}|\leq 2a\right\}}{\bf{1}}_{\left\{\theta\in\left(1/2,1\right)\right\}}\\
			&\leq (1-\theta)(4M+2k_{0}a).
		\end{aligned}
		$$
		Then,
		\begin{equation}\label{eq:A.11}
			\tag{A.11}
			\begin{aligned}
				\frac{h(x_{1})-\theta h(x_{2})}{1-\theta}&\leq\frac{|h(x_{1})- h(\theta x_{2})|}{1-\theta}+\frac{|h(\theta x_{2})- h( x_{2})|}{1-\theta}+|h(x_{2})|\\
				&\leq k_{0} \left|\frac{x_{1}-\theta x_{2}}{1-\theta}\right|+5M+2k_{0}a
				\leq k_{0} \left|\frac{x_{1}-\theta x_{2}}{1-\theta}\right|+7k_{0}a+15\phi(a).
			\end{aligned}
		\end{equation}
		Finally, in view of $f(x)=g(x)+h(x)$, inequalities (\ref{eq:A.10}) and (\ref{eq:A.11}),  we deduce that
		$$
		\begin{aligned}
			\frac{f(x_{1})-\theta f(x_{2})}{1-\theta}&=\frac{g(x_{1})-\theta g(x_{2})}{1-\theta}+\frac{h(x_{1})-\theta h(x_{2})}{1-\theta}\\
			&\leq \phi\left(\left|\frac{x_{1}-\theta x_{2}}{1-\theta}\right|+2a\right)+2k_{0}\left|\frac{x_{1}-\theta x_{2}}{1-\theta}\right|+11k_{0}a+22\phi(a).
		\end{aligned}
		$$
		Therefore, the inequality (\ref{eq:A.2}) holds.
	\end{proof}

	\begin{lemma}\label{Le:4.6}
		Let $f(x): \mathbb{R}\mapsto \mathbb{R}$ be a continuous function such that $|f(x)|\leq \phi(|x|)$ for each $x\in \mathbb{R}$, where $\phi(x):\mathbb{R}_{+}\mapsto \mathbb{R}_{+}$ is a nondecreasing continuous function.
		Assume that there exist two constants $a\geq 0$ and $k>0$  such that
		
		(i) $f(x)$ is  Lipschitz continuous on $\mathbb{R}$, i.e.,
		$$
		\forall~ (x_{1},x_{2})\in \mathbb{R} \times \mathbb{R},~~|f(x_{1})-f(x_{2})|\leq k|x_{1}-x_{2}|;
		$$
		
		(ii) $f(x)$ decreases  on $\left(-\infty,-a\right]$ and increases  on $\left[a,+\infty\right)$.\\
		Then, for each $(x_{1},x_{2})\in \mathbb{R}\times \mathbb{R}$ and $\theta \in(0,1)$, we have
		\begin{equation}\label{eq:A.12}
			\tag{A.12}
			\frac{f(x_{1})-\theta f(x_{2})}{1-\theta}\leq 4{k}\left| \frac{x_{1}-\theta x_{2}}{1-\theta}\right|+4{k}a+11\phi(a)+\phi(|x_{2}|).
		\end{equation}
	\end{lemma}

	\begin{proof}  Without lost of generality, we can assume that $a>0$ and $f(a)<f(-a)$. The proof is divided into the following three steps.
		
		{\bf  Step 1.} We construct a  continuous function $g(x)$ such that
		
		(1)$~ g(x)=f(x)$, $\forall~ x\in (-\infty,-a]\cup [a,+\infty)$;
		
		(2)$~ g(x)$ decreases  on $\left(-\infty,0\right]$ and increases  on $\left[0,+\infty\right)$.\\
		In fact, we can  define $g(x)$ as follows:

		$$
		g(x)=\begin{cases}
			f(x),& x \in (-\infty,-a]; \\
			-k_{0}(x+a)+f(-a),& x\in (-a,0];\\
			f(a), &x\in [0,a);\\
			f(x), & x\in [a,+\infty).
		\end{cases}
		$$	
		where, in view of (i),
		$$
		k_{0}:=\left|\frac{f(a)-f(-a)}{a}\right|\leq 2k.
		$$
		\begin{figure}[H]
			\centering
			\includegraphics[width=13cm,height=7cm]{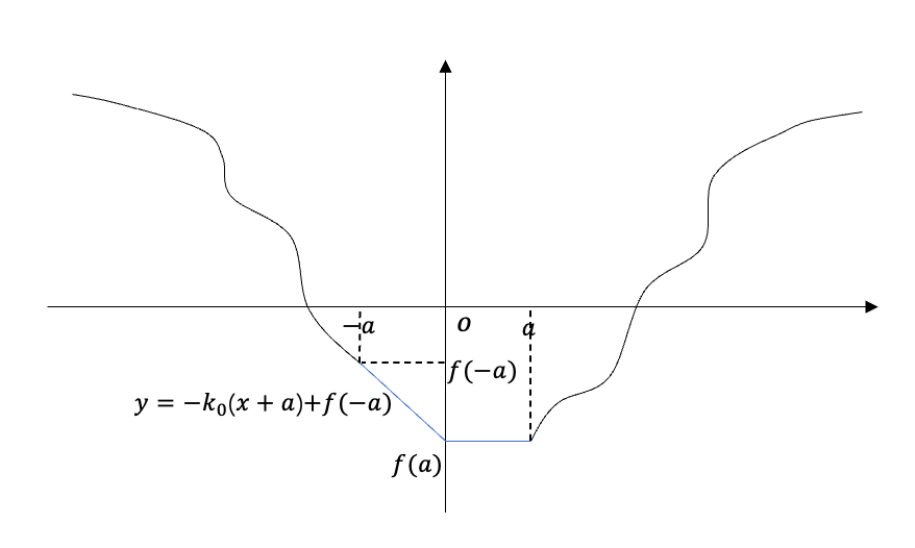}
			\caption{a graph of  $g(x)$}
			\label{5}
		\end{figure}
		\noindent We note that $g(x)$ satisfies (i) with  Lipschitz constant $2k$, i.e.,
		\begin{equation}\label{eq:A.13}
			\tag{A.13}
			\forall~ (x_{1},x_{2})\in \mathbb{R}\times \mathbb{R},~~|g(x_{1})-g(x_{2})|\leq 2k|x_{1}-x_{2}|.
		\end{equation}
		By the definition of $g(x)$, we get
		\begin{equation}\label{eq:A.14}
			\tag{A.14}
		\forall~ x\in \mathbb{R},~~~g(x)\leq f(x)+\sup\limits_{[-a,a]}|f(x)|\leq \phi(|x|)+\phi(a).
	\end{equation}

		{\bf  Step 2.}   We prove that $g(x)$ satisfies the following inequality
		\begin{equation}\label{eq:A.15}
			\tag{A.15}
			\frac{g(x_{1})-\theta g( x_{2})}{1-\theta}\leq 2k\left|\frac{x_{1}-\theta x_{2}}{1-\theta} \right|  +\phi(|x_{2}|)+\phi(a).
		\end{equation}
		Given arbitrarily $(x_{1},x_{2})\in \mathbb{R}\times \mathbb{R}$ and $\theta\in(0,1)$.  We need to consider the following six cases: $x_{1}\geq x_{2}\geq 0$, $x_{1}\leq x_{2}\leq0$,  $x_{2}\geq x_{1}\geq 0$,  $x_{2}\leq x_{1}\leq0$,    $x_{1}\leq 0\leq x_{2}$ and $x_{2}\leq 0\leq x_{1}$.

		{\bf (a) The case of $x_{1}\geq x_{2}\geq 0$:} By inequality (\ref{eq:A.13}) and (\ref{eq:A.14}), we have
		$$
		\begin{aligned}
			\frac{g(x_{1})-\theta g( x_{2})}{1-\theta}&=\frac{g(x_{1})- g( x_{2})}{1-\theta}+g(x_{2})\leq 2k\left|\frac{x_{1}-x_{2}}{1-\theta}\right|+g(x_{2})\\
			&= 2k\frac{x_{1}-\theta x_{2}+\theta x_{2}-x_{2}}{1-\theta} +g(x_{2})= 2k\frac{x_{1}-\theta x_{2}}{1-\theta}-2kx_{2} +g(x_{2})\\
			&\leq 2k\left|\frac{x_{1}-\theta x_{2}}{1-\theta} \right|  +g(x_{2})\leq 2k\left|\frac{x_{1}-\theta x_{2}}{1-\theta} \right|  +\phi(|x_{2}|)+\phi(a).
		\end{aligned}
		$$

		{\bf (b) The case of $x_{1}\leq x_{2}\leq0$:} On account of inequality (\ref{eq:A.13}) and (\ref{eq:A.14}), we have
		$$
		\begin{aligned}
			\frac{g(x_{1})-\theta g( x_{2})}{1-\theta}&=\frac{g(x_{1})- g( x_{2})}{1-\theta}+g(x_{2})\leq 2k\left|\frac{x_{1}-x_{2}}{1-\theta} \right|+g(x_{2})\\
			&= 2k\frac{-x_{1}+\theta x_{2}-\theta x_{2}+x_{2}}{1-\theta} +g(x_{2})= 2k\frac{-x_{1}+\theta x_{2}}{1-\theta}+2kx_{2} +g(x_{2})\\
			&\leq 2k\left|\frac{x_{1}-\theta x_{2}}{1-\theta} \right|+g(x_{2})\leq 2k\left|\frac{x_{1}-\theta x_{2}}{1-\theta} \right|  +\phi(|x_{2}|)+\phi(a).
		\end{aligned}
		$$
		
			{\bf (c) The case of $x_{2}\geq x_{1}\geq 0$:} In view of (2), we have
		$$
		\frac{g(x_{1})-\theta g( x_{2})}{1-\theta}=\frac{g(x_{1})- g( x_{2})}{1-\theta}+g(x_{2})\leq g(x_{2})\leq \phi(|x_{2}|)+\phi(a).
		$$
		
			{\bf (d) The case of $x_{2}\leq x_{1}\leq0$:} Using (2), we have
		$$
		\frac{g(x_{1})-\theta g( x_{2})}{1-\theta}=\frac{g(x_{1})- g( x_{2})}{1-\theta}+g(x_{2})\leq  g(x_{2})\leq \phi(|x_{2}|)+\phi(a).
		$$
		
		{\bf (e)	The case of $x_{1}\leq 0\leq x_{2}$:} It follows from $(c)$  that
		$$
		\frac{g(0)-\theta g( x_{2})}{1-\theta}
		\leq \phi(|x_{2}|)+\phi(a).
		$$
		Then, combining inequality (\ref{eq:A.13}) and the last inequality, we have
		$$
		\begin{aligned}
			\frac{g(x_{1})-\theta g( x_{2})}{1-\theta}&=\frac{g(x_{1})-g(0) }{1-\theta}+\frac{g(0)-\theta g(x_{2})}{1-\theta}\\
			&\leq 2k\left|\frac{x_{1}}{1-\theta}\right| +\phi(|x_{2}|)+\phi(a)\leq 2k\left|\frac{x_{1}-\theta x_{2}}{1-\theta} \right|+\phi(|x_{2}|)+\phi(a).\\
		\end{aligned}
		$$
		
		{\bf (f)	The case of $x_{2}\leq 0\leq x_{1}$:} According to $(d)$, it follows  that
		$$
		\frac{g(0)-\theta g(x_{2})}{1-\theta}
		\leq \phi(|x_{2}|)+\phi(a).
		$$
		Then, using inequality (\ref{eq:A.13}) and the above inequality, we have
		$$
		\begin{aligned}
			\frac{g(x_{1})-\theta g( x_{2})}{1-\theta}&=\frac{g(x_{1})-g(0) }{1-\theta}+\frac{g(0)-\theta g(x_{2})}{1-\theta}\\
			&\leq 2k\left|\frac{x_{1}}{1-\theta}\right| +\phi(|x_{2}|)+\phi(a)\leq 2k\left|\frac{x_{1}-\theta x_{2}}{1-\theta} \right|+\phi(|x_{2}|)+\phi(a).\\
		\end{aligned}
		$$
		All in all, the inequality (\ref{eq:A.15}) holds.

		{\bf Step 3.}  We prove that $f(x)$  satisfies inequality (\ref{eq:A.12}). Set $h(x):=f(x)-g(x)$. Then $h(x)$ is a Lipschitz continuous function with   Lipschitz constant $2k$, and  $h(x)\equiv 0$ when $|x|\geq a$. Note that $h(x)$ has a bound $M:=2\sup\limits_{[-a,a]}|f(x)|\leq 2\phi(a)$,  for each $(x_{1},x_{2})\in\mathbb{R} \times \mathbb{R}$ and $\theta\in (0,1)$, we have
		$$
		\begin{aligned}
			|h(\theta x_{2})-h(x_{2})|&=|h(\theta x_{2})-h(x_{2})|{\bf{1}}_{\left\{\theta\in\left(0,1/2\right]\right\}}+|h(\theta x_{2})-h(x_{2})|{\bf{1}}_{\left\{\theta\in\left(1/2,1\right)\right\}}\\
			&\leq (1-\theta)\frac{2M}{1-\theta}{\bf{1}}_{\left\{\theta\in\left(0,1/2\right]\right\}}+2(1-\theta)k|x_{2}|{\bf{1}}_{\left\{|x_{2}|\leq 2a\right\}}{\bf{1}}_{\left\{\theta\in\left(1/2,1\right)\right\}}\\
			&\leq (1-\theta)(4M+4ka).
		\end{aligned}
		$$
		Then,
		\begin{equation}\label{eq:A.16}
			\tag{A.16}
			\begin{aligned}
				\frac{h(x_{1})-\theta h(x_{2})}{1-\theta}&\leq\frac{|h(x_{1})- h(\theta x_{2})|}{1-\theta}+\frac{|h(\theta x_{2})- h( x_{2})|}{1-\theta}+|h(x_{2})|\\
				&\leq 2k \left|\frac{x_{1}-\theta x_{2}}{1-\theta}\right|+5M+4ka
				\leq 2k \left|\frac{x_{1}-\theta x_{2}}{1-\theta}\right|+4ka+10\phi(a).
			\end{aligned}
		\end{equation}
		Finally, in view of $f(x)=g(x)+h(x)$, inequalities (\ref{eq:A.15}) and  (\ref{eq:A.16}), we deduce that
		$$
		\begin{aligned}
			\frac{f(x_{1})-\theta f(x_{2})}{1-\theta}&=\frac{g(x_{1})-\theta g(x_{2})}{1-\theta}+\frac{h(x_{1})-\theta h(x_{2})}{1-\theta}\\
			&\leq 4k\left|\frac{x_{1}-\theta x_{2}}{1-\theta}\right|+4ka+11\phi(a)+\phi(|x_{2}|).
		\end{aligned}
		$$
		Therefore, the inequality (\ref{eq:A.12}) holds.
	\end{proof}

  \section*{Declarations}

    \noindent\textbf{Conflict of interest~} The authors declare that they have no known competing financial interests or personal relationships that could have appeared to influence the work reported in this paper.

    \noindent\textbf{Data Availability~} Data sharing is not applicable to this article as no new data were generated or analyzed in this study.

	
	\bigskip
	
	\end{CJK}
\end{document}